\documentclass[11pt]{article}
\usepackage{amsfonts,amsmath,amssymb,amsthm, amscd}
\usepackage[dvips]{epsfig,graphics}

\numberwithin{equation}{section}

\newtheorem{theorem}{Theorem}[section]
\newtheorem{prop}[theorem]{Proposition}
\newtheorem{lemma}[theorem]{Lemma}
\newtheorem{cor}[theorem]{Corollary}
\newtheorem{conj}[theorem]{Conjecture}

\theoremstyle{definition}
\newtheorem{definition}[theorem]{Definition}
\newtheorem{example}[theorem]{Example}
\newtheorem{remark}[theorem]{Remark}

\newcommand{\D}{\Delta}
\DeclareMathOperator{\coker}{coker}
\DeclareMathOperator{\im}{im}
\def\<{{\langle}}
\def\>{{\rangle}}

\def\a{{\alpha}}
\def\b{{\beta}}
\def\g{{\gamma}}
\def\Z{\mathbb Z}
\def\Q{\mathbb Q}
\def\R{\mathbb R}
\def\T{\mathbb T}

\def\C{\mathbb C}
\def\a{\alpha}
\def\s{\sigma}

\def\L{{\Lambda}}

\def\n{{\bf n}}
\def\m{{\bf m}}
\def\e{\epsilon}
\def\A{{\cal A}}

\begin{document}

\title{Dynamics of Twisted Alexander Invariants}

\author{Daniel S. Silver \and Susan G. Williams\thanks{Both authors partially supported by NSF grant
DMS-0706798.} \\ {\em
{\small Department of Mathematics and Statistics, University of South Alabama}}}

\maketitle 

\begin{abstract} The Pontryagin dual of the based Alexander module of a link twisted by a ${\rm GL}_N \Z$ representation is an algebraic dynamical system with an elementary description in terms of colorings of a diagram. Its topological entropy is the exponential growth rate of the number of torsion elements of twisted homology groups of abelian covers of the link exterior. 

Total twisted representations are introduced. The twisted Alexander polynomial obtained from any nonabelian parabolic ${\rm SL}_2{\mathbb C}$ representation of a 2-bridge knot group is seen to be nontrivial. The zeros of any twisted Alexander polynomial of a torus knot corresponding to a parabolic ${\rm SL}_2{\mathbb C}$ representation or a finite-image permutation representation are shown to be roots of unity.
  \noindent \end{abstract}

\begin{footnote}{Mathematics Subject Classification:  
Primary 57M25; secondary 37B40.}\end {footnote}

\section{Introduction} The Alexander polynomial $\D_k(t)$, the first knot polynomial, was a result of J.W. Alexander's efforts during 1920--1928 to compute torsion numbers $b_r$, the orders of the torsion subgroups of $H_1 M_r$, where $M_r$ is the $r$-fold cyclic cover of ${\mathbb S}^3$ branched over $k$. Here a combinatorial approach to a topological problem led to new algebraic invariants such as the Alexander module and its associated polynomials. 

As an abelian invariant, $\D_k$ can miss a great deal of information. For example, it is well known that there exist infinitely many nontrivial knots with trivial Alexander polynomial. In 1990 X.S.~Lin introduced a more sensitive invariant using information from nonabelian representations of the knot group \cite{lin}. Later, refinements of these twisted Alexander polynomials were described by M. Wada \cite{wada}, P. Kirk and C. Livingston \cite{kl}  and J. Cha
\cite{cha}.

We examine twisted Alexander modules from the perspective of algebraic dynamics. Pontryagin duality converts a finitely generated ${\mathbb Z}[t_1^{\pm 1}, \dots, t_d^{\pm 1}]$-module ${\cal M}$ into a compact abelian group $\hat {\cal M}= {\rm Hom}({\cal M}, \T)$, where $\T=\R/\Z$ is the additive circle group. Multiplication in ${\cal M}$ by $t_1, \ldots, t_d$ becomes $d$ commuting homeomorphisms of $\hat {\cal M}$, giving rise to a ${\mathbb Z}^d$-action $\s: {\mathbb Z}^d \to {\rm Aut}\ \hat {\cal M}$.  Dynamical invariants of $\s$ such as periodic point counts and topological entropy provide invariants of the module ${\cal M}$.

In previous work \cite{swTOP2}, the authors considered the case in which ${\cal M}$ is the (untwisted) based Alexander module ${\cal A}^0$ of a link $\ell$ of $d$ components and $t_1, \ldots, t_d$ correspond to meridians.  When the $d$-variable Alexander polynomial $\D_\ell$  is nonzero, its logarithmic Mahler measure coincides with the topological entropy 
$h(\s)$, a measure of the complexity of $\s$.  Furthermore, $h(\s)$ is seen to be the exponential growth rate of torsion numbers $b_\L$
associated to $\ell$, where $\L$ is a finite-index sublattice of ${\mathbb Z}^d$ expanding in all $d$ directions. 

Given a representation $\g$ from the link group $\pi= \pi_1({\mathbb S}^3 \setminus \ell)$ to ${\rm GL}_N \Z$, a based twisted Alexander module ${\cal A}^0_\g$ and twisted Alexander polynomial 
$\D_{\ell, \g}$ are defined for which analogous results hold. In particular, we give a homological interpretation for the Mahler measure of $\D_{\ell, \g}$. 

Representations to ${\rm GL}_N \Z$ arise naturally from parabolic ${\rm SL}_2{\mathbb C}$ representations as {\it total twisted} representations.  Parabolic representations of 2-bridge knot and link groups provide many examples.  We show that for any 2-bridge knot $k$ and nonabelian parabolic representation $\gamma$, the polynomial $\D_{k, \g}$ is nontrivial.

The dynamical approach here is natural for fibered knots.
In such a case, $\hat{\cal A}_\g^0$ is a finite-dimensional torus with an automorphism $\s$ determined by the monodromy. The homology eigenvalues of $\s$ are the zeroes of the twisted Alexander polynomial. We prove that any parabolic ${\rm SL}_2{\mathbb C}$ representation of a torus knot group yields a twisted Alexander polynomial with trivial Mahler measure. 

We are grateful to the Institute for Mathematical Sciences at Stony Brook University and the Department of Mathematics of the George Washington University for their hospitality and support during the fall of  2007, when much of this work was done. We thank Abhijit Champanerkar, Stefan Friedl, Jonathan Hillman, Paul Kirk,  Mikhail Lyubich and Kunio Murasugi for comments and suggestions. 

\section{Twisted dynamical colorings}

We describe twisted Alexander modules of knots and links using extended Fox colorings of diagrams. 
This unconventional approach is both elementary and compatible with the dynamical point of view that we wish to promote. 

We begin with the untwisted case for knots. Twisted homology will require only a minor modification. Later, we describe the changes that are needed for general links. 

Let ${\cal D}$ be a diagram of an oriented knot $k$ with arcs indexed by $\{0, 1,\ldots,q\}$. Recall that $\T = \R/\Z$ is the additive circle group. Regard its elements as ``colors." A {\it based dynamical coloring} of ${\cal D}$ is a labeling $\a$ of the diagram that assigns to the $i$th arc the {\it color sequence} $\alpha_i=(\alpha_{i,n})\in \T^\Z$ with  $\a_0=(\ldots, 0, 0, 0, \ldots)$
(see Figure \ref{crossing}). At each crossing, we require:
\begin{equation}\label{color1}\alpha_{i,n}+\alpha_{j,n+1}=\alpha_{k,n} + \alpha_{i, n+1}\end{equation}
for all $n\in\Z$, where the $i$th arc is the overcrossing arc and the $j$th and $k$th arcs are the left and right undercrossing arcs, respectively.  Equivalently:
$$\alpha_i+\sigma\alpha_j=\alpha_k+\sigma\alpha_i,$$
where $\s$ is the {\it shift map},  the left coordinate shift on the second subscript.

The collection of all based dynamical colorings is denoted 
by ${\rm Col}^0({\cal D})$. It is a closed subgroup of $(\T^\Z)^q\cong (\T^q)^\Z$, invariant under $\s$, which acts as a continuous automorphism. 
The pair $({\rm Col}^0({\cal D}), \s)$ is an algebraic dynamical system, a compact topological group and continuous automorphism, which we call the  {\it based coloring dynamical system} of ${\cal D}$. 

\begin{figure}
\begin{center}
\includegraphics[height=1.5 in]{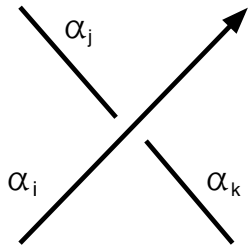}
\caption{Color sequences at a crossing}
\label{crossing}
\end{center}
\end{figure}

If ${\cal D}'$ is another diagram for $k$, then its associated dynamical system $({\rm Col}^0({\cal D}'), \s')$
is topologically conjugate to $({\rm Col}^0({\cal D}), \s)$ via a continuous isomorphism $h: {\rm Col}^0({\cal D}') \to {\rm Col}^0({\cal D})$ such that $\s \circ h = h \circ \s'$. This is the natural equivalence relation on algebraic dynamical systems. Hence $({\rm Col}^0({\cal D}), \s)$ is an invariant of $k$.

\begin{definition} The {\it based dynamical coloring system} $({\rm Col}^0(k), \s)$ of $k$ is $({\rm Col}^0({\cal D}), \s)$, where ${\cal D}$ is any diagram for $k$. \end{definition}

\begin{definition} An element $\a = (\a_{i, n})\in {\rm Col}^0(k)$ has {\it period} $r$ if $\s^r \a = \a$; equivalently, 
$\a_{i, n+r} = \a_{i, n}$ for all $i, n$. The set of all period $r$ elements is denoted by ${\rm Fix}\ \s^r$. \end{definition}

\begin{prop}\label{fixed} {\rm\cite{swTAMS}} For any knot, ${\rm Fix}\ \s$ is trivial. \end{prop}

If $\a = (\a_{i, n})\in {\rm Col}^0(k)$ has period 2, then $\a_i +\s \a_i$ is fixed by $\s$ for every $i$, and hence $\s \a_i = -\a_i$; that is, $\a_{i, n +1} = -\a_{i, n}$, for all $i, n$. Each $\a_{i, n}$ is determined by a single color $\a_{i, 0}\in \T$. The coloring condition (\ref{color1}) then becomes 
$$2 \a_{i,0} = \a_{j,0}+ \a_{k,0}.$$
Taking each $\a_{i,0}$ to be in $\Z/p$ for some positive integer $p$, 
the above becomes the well-known Fox $p$-coloring condition. Since each group $\Z/p$ embeds naturally in $\T$, we define a {\it Fox $\T$-coloring} to be any $\T$-coloring satisfying this condition. 
The group of Fox $\T$-colorings of a diagram modulo monochromatic colorings is  isomorphic to the group of period-2 points of ${\rm Col}^0(k)$.

For knots, all Fox $\T$-colorings are Fox $p$-colorings for some $p$, since $\D_k(-1)$ is nonzero and annihilates all period 2 points (see \cite{swTAMS}). 
It is well known that the group of Fox $p$-colorings of a diagram modulo monochromatic colorings is isomorphic to the homology $H_1(M_2; \Z/p)$ of the 2-fold cyclic cover with $\Z/p$ coefficients.

For each $r$, we denote by $X_r$ the $r$-fold cyclic cover of the knot exterior $X= {\mathbb S}^3\setminus k$. We denote its universal abelian cover by $X_\infty$, and regard $H_1 X_\infty$ as a finitely generated $\Z[t^{\pm 1}]$-module.  Since $H_1 X_r \cong H_1 M_r \oplus \Z$, the torsion subgroups of $H_1 X_r$ and $H_1 M_r$ are isomorphic, and we denote their order by $b_r$. We denote the rank of $H_1M_r$ (the dimension of the ${\mathbb Q}$-vector space $H_1 M_r \otimes {\mathbb Q}$) by  $\b_r$.

\begin{prop}\label{untwisted}  {\rm\cite{swTOP}} For any knot $k$, \begin{itemize}
\item[\rm(i)]  ${\rm Col}^0(k)$ is Pontryagin dual to $H_1 X_\infty$. The shift  $\s$ is dual to the meridian action of $t$. 

\item[\rm(ii)] ${\rm Fix}\ \s^r$ consists of $b_r$ tori, each of dimension $\b_r$.

\end{itemize}
\end{prop}

\begin{definition} \label{mahler} (1) The {\it logarithmic Mahler measure}  of a nonzero polynomial $p(t) = c_s t^s + \ldots +  c_1 t + c_0 \in {\mathbb C}[t]$
is 
$$m(p) = \log |c_s| + \sum_{i=1}^s \max \{ \log |\lambda_i|, 0 \},$$
where $\lambda_1, \ldots, \lambda_s$ are the roots of $p(t)$. \smallskip

(2) More generally, if $p(t_1, \ldots, t_d)$ is a nonzero polynomial in $d$ variables, then 
\begin{equation}\label{integral} m(p)= \int_0^1\cdots \int_0^1 \log|p(e^{2 \pi i \theta_1}, \ldots, e^{2 \pi i \theta_d})|\ d\theta_1\cdots d\theta_d.\end{equation}

\end{definition} 

\begin{remark} The integral in Definition \ref{mahler} can be singular, but nevertheless converges. The agreement between (1) and (2) in the case that $d=1$ is assured by Jensen's formula
(see \cite{ew} or \cite{schmidt}).
\end{remark} 

\begin{prop} \label{entropy} {\rm\cite{swTOP2}} For any knot $k$, the topological entropy $h(\s)$ is given by 
$$h(\s)=\lim_{r \to \infty} {1 \over r} \log b_r = m(\D_k(t)).$$
\end{prop}

\begin{remark} Proposition \ref{entropy} was proved earlier for the subsequence of torsion numbers $b_r$ for which $\b_r \ne 0$ by 
R. Riley \cite{riley2} and also F. Gonz\'alez-Acu\~na and H. Short \cite{gs}. The authors showed that one need not skip over torsion numbers for which $\b_r$ vanishes. They generalized the result by showing 
that the logarithmic Mahler measure of any nonvanishing multivariable Alexander polynomial $\D_\ell(t_1, \ldots, t_d)$ of a $d$-component link is again a limit of suitably defined torsion numbers. Both generalizations require a deep theorem from algebraic dynamics,  Theorem 21.1 of K. Schmidt's monograph \cite{schmidt},  which extends a theorem of D. Lind,  Schmidt and T. Ward \cite{lsw}.  \end{remark}

Consider now a linear representation $\g: \pi = \pi_1({\mathbb S}^3 \setminus k) \to {\rm GL}_N \Z$. Let $x_i$ denote the meridian generator of $\pi$ corresponding to the $i$th arc of a diagram ${\cal D}$ for $k$. We let $X_i$ denote the image matrix $\g(x_i)$. The representation $\g$ induces an action of  
$\pi$ on the $N$-torus $\T^N$.  Our {\it colors} will now be elements of $\T^N$.

\begin{definition} A {\it based $\g$-twisted dynamical coloring} of ${\cal D}$ is a labeling $\a$ of the arcs of ${\cal D}$ by {\it color sequences} $\alpha_i=(\alpha_{i,n})
\in (\T^N)^\Z$, with the 0th arc labeled by the sequence of zero vectors. The crossing condition (\ref{color1}) is replaced by:
$$\alpha_{i,n}+X_i\ \alpha_{j,n+1}=\alpha_{k,n} + X_k\ \alpha_{i, n+1}.$$
We will write this as
$$\alpha_i+X_i\ \sigma\alpha_j=\alpha_k+X_k\ \sigma\alpha_i,$$
where again $\s$ is the {\it shift map}, the coordinate shift on the $\Z$-coordinates, and $X \a_i = (X \a_{i, n})$.
\end{definition}

\begin{remark} More generally, one can consider representations $\g$ from $\pi$ to ${\rm GL}_NR$, where $\Z\subseteq R \subseteq \Q$. In this case, $\T^N$ is replaced by 
the $N$-dimensional solenoid $\hat R^N$, where $\hat R = {\rm Hom}(R, \T)$. We will say more about this  below. \end{remark}

The generalization from knots to links is natural. Let $\ell = \ell_1 \cup \cdots \cup \ell_d$ be an oriented link of $d$ components. Let ${\cal D}$ be a diagram. The abelianization $\epsilon: \pi \to \Z^d$ takes the meridian generator $x_i$ corresponding to the $i$th arc of ${\cal D}$ to a standard basis vector $\epsilon(x_i)\in \{e_1, \ldots, e_d\}$ for $\Z^d$. 

A based $\g$-twisted dynamical coloring of a diagram ${\cal D}$ is a labeling $\a$ of arcs by $\a_i =(\a_{i, \n}) \in (\T^N)^{\Z^d}$, where $\n=n_1 e_1+ \cdots + n_d e_d \in \Z^d$. For $\m \in \Z^d$, we write $\s_\m$ for the coordinate shift by $\m$ such that $\s_\m \a_i$ has $\n$th coordinate equal to $\a_{i, \n + \m}$. As before, the 0th arc is labeled with the zero sequence. At each crossing we require
\begin{equation}\label{coloring}\alpha_{i,\n}+X_i\ \alpha_{j,\n+{\epsilon (x_i)}}=\alpha_{k,\n} + X_k\ \alpha_{i, \n+{\epsilon (x_k)}}.\end{equation}
The crossing condition can be written more compactly as:
$$\alpha_i+X_i\ \sigma_{\e(x_i)}\alpha_j=\alpha_k+X_k\ \sigma_{\e(x_k)}\alpha_i.$$
The collection ${\rm Col}_\g^0({\cal D})$ of all based $\g$-twisted dynamical colorings is again a compact space, this time invariant under the $\Z^d$-action $\s=(\s_\n)$, generated by commuting automorphisms $\s_{e_1}, \ldots, \s_{e_d}$.

\begin{prop} \label{twist} Let $\ell$ be any oriented link. Up to topological conjugacy by a continuous group isomorphism, the dynamical system  $({\rm Col}^0_\g({\cal D}), \s)$ is independent of the diagram ${\cal D}$ for $\ell$. \end{prop}

 We denote this twisted coloring system by ${\rm Col}^0_\g(\ell)$, or ${\rm Col}^0_\g$ when the link $\ell$ is understood. 

Proposition \ref{twist} can be proven in a straightforward manner by checking invariance under  Reidemeister moves.   Alternatively, it follows from Theorem \ref{dual}.

The {\it topological entropy} $h(\s)$ of a $\Z^d$-action $\s: \Z^d \to {\rm Aut}\ Y$, where $Y$ is a compact metric space and $\s$ is a homomorphism to its group of homeomorphisms,  is a measure of its complexity. We briefly recall a definition from \cite{schmidt}, referring the reader to that source for details. If ${\cal U}$ is an open cover of  $Y$, denote the minimum cardinality of a  subcover by $N({\cal U})$. If $Q= \prod_{j=1}^d \{b_j, \ldots, b_j + l_j -1\} \subseteq \Z^d$ is a rectangle, let $|Q|$ be the cardinality of $Q$. Finally, let $\<Q\> = \min_{j=1, \ldots, d} l_j$. Then
$$h(\s)= \sup_{\cal U} h(\s, {\cal U}),$$
where ${\cal U}$ ranges over all open covers of $Y$, and 
$$h(\s, {\cal U}) = \lim_{\<Q\> \to \infty} {1 \over |Q|} \log N(
\vee_{\n \in Q}\ \s_{-\n}({\cal U})).$$ 
Here $\vee$ denotes common refinement of covers. Existence of the limit is guaranteed by subadditivity of $\log N$; that is,  $\log N({\cal U} \vee {\cal V}) \le \log N({\cal U}) + \log N({\cal V})$ for all open covers ${\cal U}, {\cal V}$ of $Y$.


\section{Twisted homology and colorings}  \label{twistedhomology}

Let $\pi = \< x_0, \ldots, x_q \mid r_1, \ldots, r_q \>$ be a Wirtinger presentation of $\pi = \pi({\mathbb S}^3 \setminus \ell)$, where $\ell = \ell_1 \cup \cdots \cup \ell_d$ is an oriented link, and $x_0$ is a meridian of $\ell_1$. 

As above, let $\e: \pi \to \Z^d$ be the abelianization homomorphism. We will make frequent use of the natural isomorphism from the free module $\Z^d$ to the {\sl multiplicative} free abelian group 
$\< t_1, \ldots, t_d \mid [t_i, t_j] \>$, identifying  $\n= (n_1, \ldots, n_d) \in \Z^d$  with the monomial $t^\n = t_1^{n_1}\cdots t_d^{n_d}$. 

Assume that $\g: \pi \to {\rm GL}_N R$ is a representation, where  $R$ is a Noetherian UFD. Then $V = R^N$ is a right $R[\pi]$-module. 
We regard $R[\Z^d]$ as the ring of Laurent polynomials in $t_1, \ldots, t_d$ with coefficients in $R$. 
We give $R[\Z^d] \otimes_R V$ the structure of a  right $R[\pi]$-module via
\begin{equation}\label{action} (p \otimes v)\cdot g = (pt^{\e(g)} ) \otimes (v \g(g))\ \forall g \in \pi.\end{equation}

We denote the exterior of $\ell$,  the closure of ${\mathbb S}^3$ minus a regular neighborhood of $\ell$, by $X$. Its universal cover will be denoted by $\tilde X$, and its  universal abelian cover by $X'$.

Let $C_*(\tilde X)$ denote the cellular chain complex of $\tilde X$ with coefficients in $R$. It is a free left $R[\pi]$ module, with basis obtained by choosing a  lifting for each cell of $X$. For purposes of calculation it is convenient to collapse $X$ to a $2$-complex with a single $0$-cell $*$, $1$-cells $x_0, \ldots, x_q$ corresponding to the generators of $\pi$, and $2$-cells $r_1, \ldots, r_q$ corresponding to the relators. We denote the chosen lifted cells
in $\tilde X$ by the same symbols. 

We consider the $\g$-twisted chain complex $C_*(X'; V_\g)$ of $R[\Z^d]$-modules
\begin{equation} \label{chaincplx} C_*(X'; V_\g) = (R[\Z^d]\otimes_RV) \otimes_{R[\pi]} C_*(\tilde X).\end{equation}
When it is clear what representation is being considered, we will shorten the notation $V_\g$ to $V$.  

The $\g$-twisted chain groups $C_*(X'; V)$, and hence the $\g$-twisted homology groups $H_*(X'; V)$, are finitely generated  $R[\Z^d]$-modules. The homology depends on the representation only up to conjugacy. The action of $\Z^d$ is given by 
$$t^{\n} \cdot (1 \otimes v \otimes z) =  t^{\n}   \otimes v \otimes z.$$

\begin{remark} \label{shapiro} Notation for twisted homology varies widely among authors. We have attempted to keep ours as uncluttered as possible. The notational reference to $X'$ is justified as follows. The representation $\g: \pi \to {\rm GL}_NR$ restricts to a representation of the commutator subgroup $\pi'$. Then $V$ is a right $R[\pi']$-module. Shapiro's Lemma (see \cite{brown}, for example) implies that $H_*(X'; V)$ is isomorphic to the homology groups resulting from the twisted chain complex $V \otimes_{R[\pi']} C_*(X')$.  
\end{remark}

For  $0 \le i \le q$, 
\begin{equation*}\begin{split}
\partial_1 (1 \otimes v \otimes x_i)& =  1\otimes v \otimes \partial_1 x_i
=1\otimes v \otimes (x_i\smash* -\, *)\\
& = 1\otimes v \otimes x_i\smash * -\ 1 \otimes v \otimes *.
\end{split}\end{equation*}
Using the $\pi$-action given by (\ref{action}), we can write
$$\partial_1 (1 \otimes v \otimes x_i) = t^{\e(x_i)}  \otimes vX_i \otimes * - 1 \otimes v \otimes *.$$
(We remind the reader that $X_i$ denotes $\g(x_i)$.)

If $r= x_i x_jx_i^{-1}x_k^{-1}$ is a Wirtinger relator, then one checks in a similar fashion that
$\partial_2 (1 \otimes v \otimes r)$ is given by 
\begin{equation*}\begin{split}
1 \otimes v \otimes x_i &+ t^{\e(x_i)} \otimes v X_i\otimes x_j - t^{\e(x_ix_jx_i^{-1})} \otimes v X_iX_jX_i^{-1} \otimes x_i  \\
&-t^{\e(x_ix_jx_i^{-1}x_k^{-1})} \otimes v X_iX_jX_i^{-1}X_k^{-1} \otimes x_k. 
\end{split}\end{equation*}
Since $x_ix_jx_i^{-1}x_k^{-1}$ is trivial in $\pi$, we can express the 
relation
$\partial_2 (1 \otimes v \otimes r) =0$ more simply: 
\begin{equation}\label{relations} 1 \otimes v \otimes x_i + t^{\e(x_i)} \otimes v X_i \otimes x_j = 1 \otimes v \otimes  x_k +t^{\e(x_k)} \otimes v  X_k \otimes x_i . \end{equation}

An elementary ideal $I_{\ell, \g}$ and {\it $\g$-twisted Alexander polynomial} $\D_{\ell, \g}$ are defined for $H_1(X'; V)$. The ideal $I_{\ell, \g}$ is generated by determinants of all maximal square submatrices of any presentation matrix, while $\D_{\ell, \g}$ is the greatest common divisor of the determinants, well defined up to multiplication by a unit in $R[\Z^d]$. The independence of $I_{\ell, \g}$ and $\D_{\ell, \g}$ of the choice of presentation matrix is well known
(see \cite{cf}, p. 101, for example). \bigskip

Following \cite{kl}, we write the chain module $C_1(X'; V) = W \oplus Y$, where 
$W$, $Y$ are freely generated by lifts of $x_0$ and $x_1,\ldots, x_q$, respectively. Our complex of $\g$-twisted chains can be written:
\begin{equation} \label{complex} 0 \to C_2(X'; V) \ {\buildrel \partial_2=a \oplus b \over \longrightarrow}\ 
W \oplus Y \ {\buildrel \partial_1= c+d \over \longrightarrow} \ C_0(X'; V)\to 0.\end{equation}

\begin{definition} The $\g$-{\it twisted Alexander module} ${\cal A}_\g$ is the cokernel of $\partial_2$. The {\it based} $\g$-{\it twisted Alexander module} ${\cal A}^0_\g$ is the cokernel of $b$. \end{definition}

Let $m$ be a meridianal link about the component $\ell_0$. We can regard $m$ as a subspace of the link exterior $X$, and the inclusion map induces an injection on fundamental groups. The representation 
$\g$ restricts, and we can consider the relative $\g$-twisted homology
$H_1(X', m'; V)$, where $m'$ is the lift of $m$ to the infinite cyclic cover $X'$. We can assume that cell complexes for $X$ and $m$ both have a single, common 0-cell. From the relative chain complex $C_*(X', m'; V) = C_*(X'; V)/C_*(m'; V)$, we  obtain the following homological interpretation of ${\cal A}_\g^0$. 

\begin{prop} \label{basedalex} The based $\g$-twisted Alexander module ${\cal A}_\g^0$ is isomorphic to $H_1(X', m'; V)$. \end{prop}

Although ${\cal A}_\g^0$ is an invariant of the (ordered) link by Proposition \ref{basedalex}, its dependence on the choice of component $\ell_0$ is critical. 

\begin{definition} The {\it  $\g$-twisted coloring polynomial} $D_{\ell, \g}$  is 0th elementary divisor $\D_0({\cal A}_\g^0)$. 
\end{definition}

Since ${\cal A}_\g^0$ has a square matrix presentation, $D_{\ell, \g}$ is relatively easy to compute. The quotient $D_{\ell, \g}/ {\rm det}(t^{\e(x_0)} X_0 -I)$, denoted by 
$W_{\ell, \g}$, is an invariant  introduced by M. Wada  \cite{wada}. It is well defined up to multiplication by a unit in $R[\Z^d]$. Wada's invariant is defined generally for any finitely presented group, given a homomorphism onto the integers and a linear representation. 

\begin{theorem}\label{dual} Assume that $\Z \subset R \subset \Q$.
The dual group  $\hat A_\g^0$ is topologically conjugate to the based coloring dynamical system ${\rm Col}^0_\g(\ell)$. 

\end{theorem}

\begin{proof}  Elements of the dual group $\hat A_\g^0$ are certain assignments
$$\a: t^{\n} \otimes v \otimes x_i \mapsto v \cdot \a_{i, \n} \in \T,$$ where $\a_{i, \n}$ is a vector in $\hat R^N$ for $1\le i\le q$ and $\n \in \Z^d$. Equivalently, $\a$ is a function taking each $x_i$  to an element $\a_{i, \n} \in (\hat R^N)^{\Z^d}$.  Since $\hat A_\g^0$ is the dual of the cokernel of $b$, such an assignment $\a$ is in $\hat A_\g^0$ if and only if $\a \circ b = 0$. In view of Equation (\ref{relations}),  the condition amounts to 
$$v \cdot \a_{i, \n} + vX_i \cdot   \a_{j, \n + \e(x_i)} = v \cdot \a_{k,\n} + v X_k\cdot  \a_{i, \n + \e(e_k)},$$
for every $v \in V$. Equivalently,
$$\a_{i, \n} + X_i  \a_{j, \n + \e(x_i)} =  \a_{k,\n} +  X_k  \a_{i, \n + \e(e_k)},$$
which is the condition (\ref{coloring}) for $\g$-twisted colorings. 
\end{proof}

From the chain complex (\ref{complex}), one obtains an exact sequence: 
$$0 \to H_1(X'; V) \to {\cal A}_\g^0\  \to \coker c \ {\buildrel f \over \rightarrow}\ H_0(X'; V) \to 0,$$ where the maps are natural. (This is shown in \cite{kl} in the case that $d=1$ and $R$ is a field. The general statement is proven in \cite{hln}.)

Exactness  allows us to write a short exact sequence that expresses ${\cal A}_\g^0$ as an extension of $H_1(X'; V)$
\begin{equation}\label{exactseq1}0 \to H_1(X'; V) \to  {\cal A}_\g^0 \to \ker f \to 0.\end{equation}

Pontryagin then duality induces an exact sequence
$$0 \to \widehat{\ker f} \to {\rm Col}^0_\g \to \widehat{ H_1(X'; V)} \to 0. $$

Recall that the shift map on ${\rm Col}^0_\g$ is denoted by $\s$. We denote the shifts of $\widehat{ H_1(X'; V)}$ and $\widehat{\ker f}$ by $\s'$ and $\s''$, respectively. By Yuzvinskii's Addition Formula (see Theorem 14.1 of \cite{schmidt}, for example), 
\begin{equation}\label{yuz} h(\s) = h(\s') + h(\s''). \end{equation}

The module $\ker f$ can be described more explicitly. While $H_0(X'; V)$ consists of the cosets of $C_0(X'; V)$ modulo the entire image of $\partial_1$, the module $\coker c$ is generally larger. Its elements are the cosets of $C_0(X'; V)$ modulo the image of $\partial_1$ restricted to the $\g$-twisted 1-chains generated only by $
1\otimes e_s \otimes x_0$. (As above, $e_s$ ranges over a basis of $\Z^N$.)   It follows that $\ker f\cong\im \partial_1/\im c.$

The exact sequence (\ref{exactseq1}) implies that 
\begin{equation}\label{ses1}D_{\ell, \g} = \D_{\ell, \g}\cdot \D_0(\ker f).\end{equation}
On the other hand, the short exact sequence
\begin{equation} 0 \to \ker f \to \coker c \ {\buildrel f \over \rightarrow}\ H_0(X'; V) \to 0\end{equation}
implies that 
\begin{equation}\label{ses2} \D_0(\coker c) = \D_0(\ker f) \cdot \D_0(H_0(X'; V)). \end{equation}
The elementary divisor $\D_0(\coker c)$ is  the determinant of $t^{\e(x_0)} X_0 - I$.  Note that $\D_0(H_0(X'; V)$ is a factor. 

Together equations (\ref{ses1}) and  (\ref{ses2}) imply the following relationship previously proven in \cite{kl}. 

\begin{prop} \label{alexcompute}The $\gamma$-twisted Alexander polynomial and Wada invariant are related by $$\D_{\ell, \g} = W_{\ell, \g} \cdot \D_0(H_0(X'; V)).$$ \end{prop}








In earlier work  \cite{swTOP2}, we interpreted the Mahler measure of any nonzero untwisted Alexander polynomial of a link as an exponential growth rate of torsion numbers associated to finite-index abelian branched covers of the link. Generalizing such a result for
twisted Alexander polynomials presents challenges. 
First, the representation $\g$ does not induce representations of the fundamental groups of branched covers of the link; we work with unbranched covers instead. Second, the relationship between the periodic points of the coloring dynamical system and torsion elements of the unbranched cover homology groups, which we exploited, becomes more subtle when twisting is introduced.  Third, $H_1(X'; V)$ has no obvious square matrix presentation, making the computation of the entropy $h(\s')$ more difficult. 

Let $\L \subseteq \Z^d$ be subgroup. 
An element 
$\a \in {\rm Col}_\g^0(\ell)$ is a $\L$-{\it periodic point} if 
$\s_\n \a = \a$ for every $\n \in \L$. More explicitly, if $\a$ is given by a based $\g$-twisted dynamical coloring $\a = (\a_{i, \n}) \in (\hat R^{qN})^{\Z^d}$ of a diagram for $\ell$, then  $\a$ is a $\L$-periodic point if $\a_{i, \n + \m} = \a_{i, \n}$, for every $\m \in  \L$ and $1\le i \le q$. 
We denote the subgroup of $\L$-periodic points by ${\rm Fix}_\L \s$.

Associated to the homomorphism $\pi\ {\buildrel \e \over \longrightarrow} \   \Z^d \to  \Z^d/\L$, where the second map is the natural projection, there is a regular covering space $X_\L$ of $X$ with covering transformation group $\Z^d/\L$. 
The representation $\g$ restricts to $\pi_1 X_\L$, which is a subgroup of $\pi$. (We denote the restricted representation also by $\g$.)
By Shapiro's Lemma,  $H_*(X_\L; V)$ is isomorphic to
the homology groups that result from the chain complex (\ref{chaincplx}) when $\Z^d$ is replaced by the quotient $\Z^d/\L$. 
\medskip

{\sl Henceforth, we assume that $R$ is the ring of fractions $S^{-1}\Z$ of $\Z$ by a multiplicative subset $S$ generated by a finite set of primes $p_1, \ldots, p_s$.}  The ring $R$ can be described as the set of rational numbers of the form $m/p_1^{\mu_1}\cdots p_s^{\mu_s}$, where $m\in\Z$ and $\mu_1, \ldots, \mu_s$ are nonnegative integers. The case $S=\emptyset$, whereby $S^{-1}\Z =\Z$, is included. \medskip

When the index $[\Z^d:\L] = |\Z^d/\L|$ is finite,  $H_*(X_\L; V)$ is a finitely generated $R$-module.  Since $R$ is a PID (see p.~73 of \cite{lang}, for example), we have a decomposition  \begin{equation} \label{homology} H_1(X_\L; V) \cong R^{\b_\L} \oplus R/(q_1^{n_1}) \oplus \cdots \oplus R/(q_t^{n_t}),  \end{equation} where $q_1, \ldots, q_t$ can be taken to be prime integers.  We note that $R/(q_i^{n_i}) \cong \Z/(q_i^{n_i}) \otimes_\Z R$.  

The following lemma is well known. (See p.~312 of \cite{dk}, for example.)

\begin{lemma} \label{order}  For prime $q$, $$\Z/(q^{n}) \otimes_\Z R = \begin{cases} \Z/(q^{n}) & \text{if $q \notin S$},\\ 0& \text{if $q \in S$}.\end{cases}$$ \end{lemma}

In view of Lemma \ref{order} we may as well assume that no prime $q_i$ in the decomposition (\ref{homology}) is in $S$.

\begin{definition} The $\g$-{\it twisted $\L$-torsion number} $b_{\L, \g}$ is the order of the $\Z$-torsion subgroup of $H_1(X_\L; V)$. When $d=1$ and $\L=r\Z$, we write $b_{r,\g}$.
\end{definition}

Since $R^{\b_\L}$ is $\Z$-torsion-free, we have $b_{\L,\g} = q_1^{n_1}\cdots q_t^{n_t}$. By additivity of the Pontryagin dual operator and the well-known fact that the dual of any finite abelian group is isomorphic to itself, we have
$$\widehat{H_1(X_\L; V)} \cong \hat R^{\b_\L} \oplus \Z/(q_1^{n_1}) \oplus \cdots \oplus \Z/(q_t^{n_t}).$$
If $R=\Z$, then $\hat R$ is the circle. Otherwise, $\hat R$ is a solenoid; it is connected but not locally connected. 

\begin{cor}\label{components} The number of connected components of $\widehat{H_1(X_\L; V)}$ is equal to the $\g$-twisted $\L$-torsion number $b_{\L, \g}$. \end{cor}


For any subgroup $\L$ of $\Z^d$, define $\<\L\>$ to be $\min\{ |\n| : 0 \ne \n \in \L\}$. Our first theorem on growth rates of twisted torsion numbers assumes that $R= \Z$ (that is, $S$ is empty). The second requires that $d=1$, but allows $S$ to be nonempty. We recall that 
$m(p)$ is the logarithmic Mahler measure of $p$. We extend to rational functions by defining $m(p/q) = m(p)-m(q).$ We recall also that the shift $\s''$ was defined following (\ref{exactseq1}).

\begin{theorem} \label{mainthm}  Let $\ell = \ell_1 \cup \cdots \cup \ell_d\subseteq {\mathbb S}^3$ be an oriented link and $\g: \pi \to {\rm GL}_N \Z$ a representation. Assume that $D_{\ell, \g} \ne 0$. Then
\begin{itemize} \item[\rm(1)] $$  \limsup_{\<\L\>\to \infty} {1 \over |\Z^d/\L|} \log b_{\L, \g} = h(\s) - h(\s'');$$
\item[\rm(2)] 
$$m(W_{\ell, \g})  \le  \limsup_{\<\L\>\to \infty} {1 \over |\Z^d/\L|} \log b_{\L, \g} \le m(D_{\ell, \g}).$$
\end{itemize}
Here the limit is taken over all finite-index subgroups $\L$ of $\Z^d$, and 
$\limsup$ can be replaced by an ordinary limit when $d=1$. 

 \end{theorem}

\begin{cor} \label{mainthmcor} If $H_1(X'; V)$ has a square presentation matrix or if the spectral radius of $X_0$ does not exceed 1, then $$ m(\D_{\ell, \g})=\limsup_{\<\L\>\to \infty} {1 \over |\Z^d/\L|} \log b_{\L, \g}.$$ \end{cor}

If we restrict to the case of a knot $k$ (that is, $d=1$), then we can say more. 

\begin{cor}  \label{mainthmknots} Assume that $k \subset {\mathbb S}^3$ is a knot, and $\g: \pi \to {\rm GL}_N\Z$ is a representation such that $\D_{k, \g} \ne 0$. Then 
$$m(\D_{k, \g}) = \lim_{r \to \infty} {1 \over r} \log b_{r, \g}.$$ \end{cor}


Corollary \ref{mainthmknots} generalizes for the rings $S^{-1}\Z$. First, we express $D_{k, \g}$ uniquely as $d \tilde D_{k, \g}$, where $d \in R = S^{-1}\Z$ and $\tilde D_{k, \g}$ is a primitive polynomial in $\Z[t^{\pm 1}]$. Express $\D_{k, \g}$ similarly as $\delta \tilde \D_{k, \g}$. By Gauss's Lemma, $\tilde \D_{k, \g}$ divides $\tilde D_{k, \g}$. Moreover, $\delta$ divides $d$.

 \begin{theorem} \label{thm2} If $d$ is a unit in $R$, then $$m(\tilde \D_{k, \g})  =  \lim_{r \to \infty} {1 \over r} \log b_{r, \g}.$$ \end{theorem}

 \begin{proof} [Proof of Theorem \ref{mainthm}] We regard the basepoint $*$ as a subspace of $X$. 
 The representation $\g: \pi \to {\rm GL}_N \Z$ restricts trivially, and 
 hence we can consider the $\g$-twisted chain complex
 $C_*(*'; V)=\{C_0(*'; V)\}$, defined as in (\ref{chaincplx}), where $*'$ denotes the preimage of $* \in X$ under the projection map 
 $X' \to X$. 
 
 Clearly, $C_0(*'; V) \cong \Z[\Z^d]\otimes_\Z V$. Since $C_*(*'; V)$ is a sub-chain complex of $C_*(X'; V)$, the quotient
chain complex $C_*(X', *'; V) = C_*(X'; V)/ C_*(*'; V)$ is defined. From it, relative $\g$-twisted homology groups $H_*(X', *'; V)$ are defined. It is immediate that $H_1(X', *'; V) \cong {\cal A}_\g$.

Let $J$ be the ideal of $\Z[\Z^d]$  generated by all elements of 
$\L $. Define a twisted chain complex  of $\Z[\Z^d]$-modules $$JC_*(X',*'; V) = (J \otimes_\Z V) \otimes_{\Z[\pi]}C_*(\tilde X).$$ Then  $JC_*(X',*'; V)$ can be considered as a subcomplex of $C_*(X',*'; V)$. We denote the quotient chain complex by $\bar C_*(X',*'; V)$. Shapiro's Lemma implies that the homology groups of the latter chain complex are isomorphic to $H_*(X_\L, *_\L; V)$, where $*_\L$ denotes the preimage of $* \in X$ under the projection map $X_\L \to X$. 

The short exact sequence of chain complexes $$0 \to JC_*(X',*'; V) \to C_*(X', *'; V) \to \bar C_*(X',*'; V) \to 0$$ gives rise to a long exact sequence of homology groups: 
$$\cdots \to H_1( JC_*(X', *'; V)) {\buildrel i \over \longrightarrow}\ {\cal A}_\g \to H_1(X_\L, *_\L; V) \to 0.$$
Since every chain in $JC_1(X', *'; V)$ and $C_1(X', *'; V)$ is a cycle, it follows that the image of $i$ is $J{\cal A}_\g$. Hence by exactness,
\begin{equation} \label{homology1} H_1(X_\L, *_\L; V) \cong {\cal A}_\g/J{\cal A}_\g.\end{equation}

Now consider the short exact sequence of chain complexes $$0 \to \bar C_*(*'; V) \to \bar C_*(X'; V) \to \bar C_*(X',*'; V) \to 0,$$ 
where $\bar C_*(*';  V)$ and $\bar C_*(X'; V)$ are defined as $\bar C_*(X',*'; V)$ was, by replacing the $\Z[\Z^d]$ with  $J$. The sequence gives rise to a long exact sequence of homology groups:
\begin{equation}\label{homology2} 0 \to H_1(X_\L; V) \to H_1(X_\L,*_\L; V) \ {\buildrel \bar\partial \over \longrightarrow}\ {\rm Im}\ \bar\partial \to 0. \end{equation}
The image of $\bar \partial$ is an $\Z$-submodule of $\bar C_0(*'; V) \cong J \otimes_\Z V$. The latter module is isomorphic to the direct sum of $|\Z^d/\L|$ copies of $V\cong \Z^N$, and therefore is free. Since $\Z$ is a PID, the submodule ${\rm Im}\ \bar\partial$ is also free, and the sequence (\ref{homology2}) splits. We have
$${\cal A}_\g/J{\cal A}_\g \cong H_1(X_\L; V)\oplus {\rm Im}\ \bar\partial.$$

We obtain ${\cal A}_\g^0/J{\cal A}_\g^0$ by killing the coset of $1\otimes v \otimes x_0 \in {\cal A}_\g^0$. In view of the splitting, this can be done by killing
its image in ${\rm Im}\ \bar\partial$, resulting in a summand that is easily seen to be isomorphic to ${\rm ker} f/J\,{\rm ker} f$. We have 
$${\cal A}_\g^0/J{\cal A}_\g^0 \cong H_1(X_\L; V) \oplus {\rm ker} f/J\,{\rm ker} f.$$
Pontraygin duality gives
\begin{equation}\label{product} {\rm Fix}_\L \s \cong \widehat{H_1(X_\L; V)}\oplus \widehat{{\rm ker}f/J\,{\rm ker} f}.\end{equation}

The decomposition (\ref{product}) implies that the number of connected components of  ${\rm Fix}_\L\s$ is the product of the numbers of connected components of  $\widehat{H_1(X_\L; V)}$ and $\widehat{{\rm ker}f/J\,{\rm ker} f}$. 

Theorem 21.1 of \cite{schmidt} implies that the exponential growth rate of the number of connected components of ${\rm Fix}_\L\s$   equal to $h(\s)$, while that of $\widehat{{\rm ker}f/J{\rm ker} f}$ is 
$h(\s'')$.  The number of connected components of $\widehat{H_1(X_\L; V)}$ is $b_{\L, \g}$, by Corollary \ref{components}. Hence equation (\ref{yuz}) completes the proof of the first statement of 
Theorem \ref{mainthm}.

In order to prove the second statement of Theorem \ref{mainthm}, 
we note that ${\cal A}_\g^0$ has a square presentation matrix with determinant $D_{\ell, \g}$. Example 18.7 of \cite{schmidt} implies that $h(\s) = m(D_{\ell, \g})$. Furthermore, ${\rm ker} f$ is a submodule of $\coker c$, which also has a square presentation matrix $t^{\e(x_0)} X_0 - I$. Since $\widehat{{\rm ker} f}$ is a quotient of $\widehat{\coker c}$,  the topological entropy $h(\s'')$ of its shift does not exceed that of $\widehat{\coker  c}$. The latter is
$m({\rm det}(t^{\e(x_0)} X_0 - I))$, again by Example 18.7 of \cite{schmidt}. Since $m(W_{\ell, \g})= m(D_{\ell, \g}) - m({\rm det}(t^{\e(x_0)} X_0 - I))$, the second statement is proved. 
 \end{proof}

 \begin{proof}[Proof of Corollary \ref{mainthmcor}] If $H_1(X'; V)$ has a square presentation matrix, then the desired conclusion follows immediately from Example 18.7 of \cite{schmidt}. 
 
 The logarithmic Mahler measure of ${\rm det}(t^{\e(x_0)}X_0 -I)$ is equal to that of ${\rm det}(t^{\e(x_0)}-X_0)$, by a simple change of variable in (\ref{integral}). The latter vanishes if no eigenvalue of $X_0$ has modulus greater than 1. In this case, $h(\s') = h(\s)$ by (\ref{yuz}). The latter is equal to $m(D_{\ell, \g})$ by Example 18.7 of \cite{schmidt}. However, in view of (\ref{ses1}), we have  $m(D_{\ell, \g})= m(\D_{\ell, \g})$. 
 
 \end{proof}
 
\begin{proof}[Proof of Corollary \ref{mainthmknots}] By Corollary \ref{mainthmcor} it suffices to show that  $H_1(X'; V)$ has a square presentation matrix. This follows from the fact that a finitely generated torsion $\Z[t^{\pm 1}]$-module has a square presentation matrix if and only if it has no nonzero finite submodule (see page 132 of \cite{hillman}). The hypothesis that $\D_{k, t} \ne 0$ implies that $D_{k, t} \ne 0$, and hence the based $\g$-twisted Alexander module $\A_\g^0$ is a finitely generated $\Z[t^{\pm 1}]$-torsion module. Since $\A_\g^0$ has a square presentation matrix, it has no nonzero finite submodule. The $\g$-twisted homology group $H_1(X'; V)$ is a finitely generated submodule of $\A_\g^0$, and so it too has no nonzero finite submodule. Hence $H_1(X'; V)$  has a square presentation matrix. 

 \end{proof}

 \begin{proof}[Proof of Theorem \ref{thm2}]

 The $R[t^{\pm 1}]$-module ${\cal A}_\g^0$ has a square presentation matrix, and the greatest common divisor of the coefficients of its determinant $D_{k, \g}$ is a unit in $R$. By Theorem (1.3) of \cite{crowell}, ${\cal A}_\g^0$ is $\Z$-torsion free.
 Since $H_1(X'; V)$ is a submodule of ${\cal A}_\g^0$, by the exact sequence (\ref{exactseq1}), it is $\Z$-torsion free as well. Hence $H_1(X'; V)$ embeds naturally in $H_1(X'; V) \otimes_\Z\Q$. 
 
 Since $\Q[t^{\pm 1}]$ is a PID, there exist primitive polynomials $f_1, \ldots, f_k \in \Z[t]$ such that $f_j$ divides $f_{j+1}$, for $j=1, \ldots, k-1$, and 
 $$H_1(X'; V)\otimes_\Z\Q \cong \Q[t^{\pm 1}]/(f_1) \oplus\cdots \oplus \Q[t^{\pm 1}]/(f_k) : = {\cal M}.$$
 The product $f_1\cdots f_k$ is equal to $\tilde \D_{k, \g}$, up to multiplication by $\pm t^\n$. Generalizing the argument of Lemma 9.1 
 of \cite{schmidt}, we see that $H_1(X'; V)$ embeds with finite index in 
 ${\cal N} = R[t^{\pm 1}]/(f_1)\oplus \cdots \oplus R[t^{\pm 1}]/(f_k).$
 
 Set ${\cal N}' = \Z[t^{\pm 1}]/(f_1) \oplus \cdots \oplus \Z[t^{\pm 1}]/(f_k)$. Let $b_r'$ be the number of connected components of  period-r points of the dual dynamical system. Lemma 17.6 of \cite{schmidt} implies
 $$m(f_1\cdots f_k) = \lim_{r\to \infty}{1 \over r} \log b_r'.$$
 Since ${\cal N}' \subset {\cal N} \subset {\cal M}$, and since the systems dual to ${\cal N}$ and ${\cal M}$ have the same entropy by Lemma 17.6 of \cite{schmidt}, we have $$h(\s) = \lim_{r\to \infty}{1 \over r} \log b_r'.$$ 
 Since ${\cal N} = {\cal N}' \otimes_\Z R$, we see from Lemma \ref{order} that $b_{r, \g}$ is the largest factor of $b_r'$ not divisible by any prime in $S$. It was shown in \cite{riley2} (see also 
\cite{swENS}) that for any prime $p$, 
$$  \lim_{r\to \infty}{1 \over r} \log b_r'^{(p)} =0,$$
where $ b_r'^{(p)}$ is the largest power of $p$ dividing $b_{r, \g}$. Thus
$b_{r, \g}$ has the same growth rate as $b_r'$, and this rate is $m(\tilde \D_{k, \g})$.
 
 \end{proof}

\section{Structure of twisted Alexander modules for knots} 

Assume that $k \subset {\mathbb S}^3$ is a knot and $\g: \pi \to {\rm GL}_NR$ is a representation of its group. As before $R = S^{-1}\Z$, where $S$ is a multiplicative subset of $\Z$ generated by finitely many primes, possibly empty. As in the previous theorem, we assume that $D_{k, \g}$ is primitive, and hence ${\cal A}_\g^0$ is $\Z$-torsion free. 

We show that the based $\g$-twisted Alexander module $\A_\g^0$ is completely described by a pair of embeddings $f,g: R^M \to R^M$. This simple description appears to be well known, at least in the untwisted case. However, we have not found a reference for it. 

From the description, the twisted Alexander polynomial is easily found. Also, the decomposition of entropy $h(\s)$ into $p$-adic and Euclidean parts \cite{lw} can be seen from this perspective. 

Recall that  $\A_\g^0$ is generated as an $R[t^{\pm 1}]$-module by the $1$-chains $t^n \otimes v\otimes x_i$, where $n \in \Z$, $v$ ranges over a fixed basis for $V$, and $i$ ranges over all arcs of a diagram ${\cal D}$ for $k$ except for a fixed arc. A defining set of relations, corresponding to the crossings of ${\cal D}$, is given as in (\ref{relations})
by 
\begin{equation}\label{rels}t^n\otimes v \otimes x_i + t^{n+1}\otimes vX_i \otimes x_j = t^n \otimes v \otimes x_k + t^{n+1}\otimes vX_k \otimes x_i.\end{equation}

Define $B$ to be the $R$-module with generators  $t^n  \otimes v\otimes x_i$, as above but with $n=0,1$. Relations are those of (\ref{rels}) with $n=0$. Let $U$ be the $R$-module freely generated by $1$-chains $1 \otimes v \otimes x_i$. Let $f: U \to B$ 
be the obvious  homomorphism. Define $g: U \to B$ to be the homomorphism sending each $1 \otimes v \otimes x_i$ to $t \otimes v \otimes x_i$.

If either $f$ or $g$ is not injective, then replace $U$ by the quotient $R$-module $U/(\ker f+\ker g)$. Replace $B$ by $B/(f(\ker g)+g(\ker f))$,
and $f,g$ by the unique induced homomorphisms. If again $f$ or $g$ fails to be injective, then we repeat this process. Since $R$ is Noetherian, we obtain injective maps $f,g$ after finitely many iterations.

\begin{lemma} \label{structure}  The module $\A_\g^0$ is isomorphic to the infinite sum
$$\cdots  \oplus_{R^M} R^M \oplus_{R^M} R^M \oplus_{R^M} \cdots,$$
with identical amalgamations given by the monomorphisms $R^M\ {\buildrel g  \over \leftarrow}\  R^M\  {\buildrel f \over \rightarrow}\  R^M$.  Moreover, $\D_{k, \g}(t)= {\rm det} (g-t f)$.

\end{lemma}
 
\begin{proof} For $n \in \Z$, let $B_n$ be the $R$-module generated by $1$-chains $t^n \otimes v \otimes x_i$ and $t^{n + 1} \otimes v \otimes x_i$, and relations given by (\ref{rels}). (As above, $v$  ranges over a basis of $V$ and $i$ ranges over all arcs of ${\cal D}$ except the fixed arc.) We regard the generators $t^n \otimes v \otimes x_i\in B_{n -1}$ as distinct from  
$t^n \otimes v \otimes x_i \in  B_n$. 

Let $U_n$ be the $R$-module freely generated by the $1$-chains  $t^n \otimes v \otimes x_i$. 

Define homomorphisms $f_n: U_n \to B_n$  by 
$t^n \otimes v \otimes x_i\mapsto t^n \otimes v \otimes x_i.$ 
Define $g_n: U_n \to B_{n-1}$ by $t^n \otimes v \otimes x_i\mapsto t^n \otimes v \otimes x_i$. It is clear that $\A_\g^0$ is isomorphic to the infinite amalgamated sum 
\begin{equation}\label{structure} \cdots \oplus_{U_n} B_n \oplus_{U_{n +1}} B_{n+1} \oplus_{U_{n+2}} \cdots, \end{equation}
where the amalgamation maps $B_{n-1}\ {\buildrel g_n \over \leftarrow}\  U_n\  {\buildrel f_n\over \rightarrow}\  B_n$ are not necessarily injective. 
However, the image under $f_n$ (resp. $g_n$) of any element that is in the kernel of $g_n$ (resp. $f_n$) is trivial in $\A_\g^0$. Hence we can apply the operation described prior to the statement of Lemma \ref{structure} simultaneously for all $n$, to ensure that the maps $f_n$ and $g_n$ are injective.

In any amalgamated free product $G= B*_UB'$, with amalgamating maps $B\ {\buildrel g \over \leftarrow}\  U\  {\buildrel f \over \rightarrow}\  B$, the natural maps $B \to G$ and $B' \to G$ are injections. The analogous statement holds for amalgamated sums $B \oplus_UB'$ in which the amalgamating maps are injective. The proof is easy and left to the reader. Consequently, each $B_n$ is naturally embedded in $\A_\g$. 

Since $\A_\g^0$ is $\Z$-torsion free,  $B_n$ are finitely generated free $R$-modules isomorphic to $R^M$ for some positive integer $M$, independent of $n$. 

The decomposition (\ref{structure}) induces an $R[t^{\pm 1}]$-module structure, with generators $b_1, \ldots, b_M \in B$ corresponding to module generators, and the amalgamation $B\ {\buildrel g \over \leftarrow}\  U\  {\buildrel f\over \rightarrow}\  B$ inducing a defining set of module relations $g(u_i)-t f(u_i),$ where $u_1, \ldots, u_M$ generate $U$. Clearly $\D_{k, \g} = {\rm det}(g-tf)$. 
 \end{proof}

\begin{remark} \label{coefficient} If $\D_{k, \g} \ne 0$, then after multiplication by a unit in $R[t^{\pm 1}]$, we can assume that $\D_{k, \g}$ is an integer polynomial with with leading and trailing coefficients 
$c_1, c_0$ not divisible by any primes in $S$. Then   $|c_1| =  |R^M: g(R^M)|$ and $|c_0|= |R^M: f(R^M)|$. 

When $c_1=c_0= 1$, the maps
$f$  and $g$  are isomorphisms. In this case, $\A_\g^0$ is isomorphic to $R^M$ with the action of $t$ given by $g\circ f^{-1}$. 
This is the case whenever $k$ is a fibered knot
(see \cite{cha}, \cite{swJKTR}). If, in addition, $R=\Z$, then $\sigma$ is a toral automorphism and its entropy is the sum of the logs of moduli of eigenvalues outside the unit circle, corresponding to expansive directions on the torus.

For non-fibered knots $k$, the coefficients $c_0, c_1$ measure how far $f, g$ are from being isomorphisms, and, in this sense, provide a measure of the non-fiberedness of $k$. The map $\sigma$ is an automorpism of a solenoid, and the contribution of the leading coefficient to topological entropy reflects expansion in p-adic directions \cite{lw}. Work of S. Friedl and S. Vidussi \cite{fv} suggests that $c_0$ and $c_1$ are complete fibering obstructions with $R = \Z$. 

\begin{conj} A knot $k$ is fibered iff for every representation
$\g: \pi \to {\rm GL}_N \Z$,  both $c_0$ and $c_1$ have absolute value 1.  \end{conj} 

\end{remark}

\section{Examples} \label{examples}

Any finite group can be regarded as a subgroup of permutation matrices in ${\rm GL}_N \Z$, for some $N$. Hence finite-image representations of a link group $\pi$ are a natural source of examples of representations $\g: \pi \to {\rm GL}_N\Z$ for which the techniques above can be applied.  Examples can be found throughout the literature (see  \cite{fv2}, \cite{hln}, \cite{swAGT},  for example). 

Here we focus on a less obvious source of examples, coming from representations  $\g$ of $\pi$ in ${\rm SL}_2 \C$. Following \cite{hs}, we say that $\g$ is {\it parabolic} if the image of any meridian is a matrix with trace 2. The terminology is motivated by the fact that $\g$ projects to a representation $\pi \to {\rm PSL}_2 \C = {\rm SL}_2 \C/\<-I\>$ sending each meridian to a parabolic element. A theorem of Thurston \cite{thurston} ensures that if $k$ is a hyperbolic knot, then the discrete faithful representation $ \pi \to {\rm PSL}_2 \C$ describing the hyperbolic structure of ${\mathbb S}^3\setminus k$ lifts to a parabolic representation $\g: \pi \to {\rm SL}_2 \C$.

Let $k$ be a $2$-bridge knot. Such knots are parameterized by pairs of relatively prime integers $\a, \b$ such that  $\b$ is odd and $0< \b < \a$. The knots $k(\a, \b)$  and $k(\a', \b')$ have the same type if and only if $\a = \a'$ and  $\b' = \b$ or $\b\b' \equiv 1 \mod \a$. (See \cite{mur}, for example.) 

The group of $k(\a, \b)$ has a presentation of the form $\pi= \< x, y \mid W x = y W\>$, where $x, y$ are meridians and $W$ is a word in $x^{\pm 1}, y^{\pm 1}$. As Riley showed in \cite{riley1}, any nonabelian parabolic representation $\g: \pi \to {\rm SL}_2 \C$  is conjugate to one that maps 

\begin{equation} x\mapsto\begin{pmatrix} 1 & 1 \\  0 & 1\end{pmatrix},\ y\mapsto \begin{pmatrix} 1 & 0 \\ w & 1\end{pmatrix} \label{standard}\end{equation} 
for some $0 \ne w \in \C$. In fact $w$ is an algebraic integer, a zero of a monic integer polynomial $\Phi_{\a, \b}(w)$ that is the $(1,1)$-entry of the matrix $\g(W)$. We call $\Phi_{\a, \b}(w)$ a {\it Riley polynomial} of $k$. 

Similar ideas apply to 2-bridge links, which have presentations of the form $\<x, y \mid Wy=yW\>$  (see, for example, Section 4.5 of  \cite{mr}). 

\begin{remark} It is clear from what has been said that a 2-bridge knot $k(\a, \b)$ has at most two Riley polynomials $\Phi_{\a, \b}, \Phi_{\a, \b'}$, where $\b \b' \equiv 1 \mod \a$. Either can be used to determine the set of conjugacy classes of nonabelian parabolic representations of the knot group. \end{remark} 

Let $\phi(w)$ be an irreducible factor of $\Phi_{\a, \b}(w)$. We define the {\it total representation} $\g_\phi: \pi \to {\rm GL}_N\Z$ to be the representation determined by 

$$x\mapsto\begin{pmatrix} I & I \\  0 & I\end{pmatrix},\ y\mapsto \begin{pmatrix} I & 0 \\ C & I\end{pmatrix},$$
where $C$ is the companion matrix of $\phi$. Note that $N$ is $2 \cdot \deg \phi$. The corresponding twisted invariant $\D_{k, \g_\phi}$ is the {\it total $\g_\phi$-twisted Alexander polynomial of} $k$. Like the classical Alexander polynomial, it is well defined up to multiplication by $\pm t^i$. Our terminology is motivated by the fact, easily proved, that $\D_{k, \g_\phi}$ is the product of the $\D_{k, \g}$ as we let $w$ range over the roots of $\phi$. (This observation, which is a consequence of \cite{ksw}, yields a useful strategy for computing total twisted Alexander polynomials when the degree of $\phi$, and hence the size of the companion matrix, is large.) 

\begin{remark} \label{companion} (1) For $2$-bridge knots, the algebraic integer $w$ is in fact an algebraic unit \cite{riley1}.  Hence the trailing coefficient of $\Phi_{\a, \b}(w)$ is $\pm 1$, and so the companion matrix $C$ is unimodular. However, this need not be true for $2$-bridge links (see Example \ref{whiteheadex}).

(2) The notions of total representation and total twisted Alexander polynomial generalize easily. Assume that $\g: \pi \to {\rm SL}_2 \bar\Q$ is a representation  with image $\Z[w_1, \ldots, w_n]$. (Up to conjugation, the parabolic ${\rm PSL}_2 \C$ representation describing the complete hyperbolic structure of a knot complement projects to such a representation \cite{thurston}.) Let 
$\phi_1, \ldots, \phi_n$ be the minimal polynomials of the algebraic numbers $w_1, \ldots, w_n$ with leading coefficients $c_1, \ldots, c_n$. Over the ring $R=\Z[1/lcm(c_1\cdots c_n)]$, the polynomials are monic, and we can consider their companion matrices. In order to define the {\it total representation} of $\g$ we embed ${\rm SL}_2(\Z[w_1, \ldots, w_n])$ in ${\rm GL}_N R$ by  replacing $2 \times 2$ matrices with appropriate block matrices, as above. Here $w_i$ is replaced by a block diagonal matrix of the form $I\oplus \cdots \oplus C_i \oplus \cdots \oplus I$ and $N = 2 \Sigma_i \deg \phi_i$. In practice, it might be possible to find a more economical embedding. 
\end{remark}

For any total representation $\g_\phi$ associated to a nonabelian parabolic representation $\g: \pi \to {\rm SL}_2 \C$, the boundary homomorphism  
$C_1(X'; V)$ $ {\buildrel \partial_1 \over \longrightarrow}\ 
 C_0(X'; V)$ sends the $1$-chain $1 \otimes v \otimes x- 1 \otimes v \otimes y$ to $$t\otimes v(\g_\phi(x)-\g_\phi(y))\otimes * =
t \otimes v \begin{pmatrix} 0 & I \\-C & 0\end{pmatrix}\otimes *.$$
Such chains generate $C_0(X'; V)$, and hence $H_0(X'; V)$ vanishes. Since $\D_0(\coker c) = {\rm det}\ (t X - I) = (t-1)^{2  \deg \phi}$, Proposition \ref{alexcompute} implies that the total $\g_\phi$-twisted Alexander polynomial of $k$ is
\begin{equation} \label{compute}\D_{k, \g_\phi} = \D_0({\cal A}_\g^0)/(t-1)^{ 2 \deg \phi}.\end{equation}

\begin{example}\label{trefoil} The group of the trefoil knot $k=3_1$
has presentation $\pi = \<x, y \mid xyx=yxy\>$ and Riley polynomial $\Phi_{3,1}(w) = w+1$. Up to conjugation, $\pi$ admits a single nonabelian parabolic representation,
$$x\mapsto\begin{pmatrix} 1 & 1 \\  0 & 1\end{pmatrix},\ y\mapsto \begin{pmatrix} 1 & 0 \\ -1 & 1\end{pmatrix}.$$
Using Equation \ref{compute} the total $\g_\Phi$-twisted Alexander polynomial of $k$ is seen to be $t^2+1$.  

The based $\g_\Phi$-twisted Alexander module ${\cal A}_\g^0$ is isomorphic to 
$$\< 1\otimes v \otimes y \mid t \otimes vX \otimes y = 1\otimes v \otimes y + t^2 \otimes vYX \otimes y\>,$$
where $v$ ranges over a basis for $\Z^2$. The relations can be replaced by 
$$t^2 \otimes v \otimes y = t\otimes v \bar X\bar Y X\otimes y - 1\otimes v \bar X \bar Y  \otimes y,$$
where \ $\bar {}$ \ denotes inversion. 
As a $\Z[t^{\pm 1}]$-module, $\A_{\g_\Phi}^0$ is isomorphic to $\Z^4$ with action of $t$ given by right multiplication  by the matrix 
$$M=\begin{pmatrix} 0 & I\\  -\bar X\bar Y& \bar X\bar Y X\end{pmatrix}=\begin{pmatrix} 0 & 0 &1&0\\  0 & 0 &0&1\\ 0 & 1 & 0 & -1\\ -1&-1&1& 2\end{pmatrix}.$$
The dual group $\hat \A^0_{\g_\Phi}$, which is topologically conjugate to ${\rm Col}_{\g_\Phi}^0$, is a 4-torus $\T^4$ with shift $\s$ given by left multiplication by $M$, and entropy zero. We discuss general torus knots in Section 7. 

We remark that the (untwisted) coloring dynamical system ${\rm Col}^0$ is a 2-torus $\T^2$ with shift 
$$\begin{pmatrix} 0 & 1 \\ -1 & 1 \end{pmatrix}.$$
\end{example} 

\begin{example} \label{figure8} We consider another fibered knot, the figure-eight knot $k= 4_1$ (``Listing's knot"). Its group has presentation $\pi= \<x, y \mid y x \bar y x y= x y \bar x y x \>$ and Riley polynomial $\Phi_{5,3}(w)=w^2-w+1$. Up to conjugation, $\pi$ admits two nonabelian parabolic representations
$$x\mapsto\begin{pmatrix} 1 & 1 \\  0 & 1\end{pmatrix},\ y\mapsto \begin{pmatrix} 1 & 0 \\ w & 1\end{pmatrix},$$
where $w = (1 \pm \sqrt {-3})/2$. (The representation with $w=(1 - \sqrt {-3})/2$ projects to the discrete faithful representation corresponding to the complete hyperbolic structure on the figure-eight knot complement.) 
Using (\ref{compute}), the total $\g_\Phi$-twisted Alexander polynomial of $k$ is seen to be $(t^2-4 t+1)^2$.  

The dynamical system ${\rm Col}_{\g_\Phi}^0$ is an 8-torus $\T^8$. The shift $\s$ is described by 
$$M=\begin{pmatrix} 0 & 0 & 0 & 0 &  1 & 0 & 0 & 0\\
                                  0 & 0 & 0 & 0  & 0 & 1 & 0 & 0\\
                                  0 & 0 & 0 & 0 &  0 & 0 & 1 & 0 \\
                                  0 & 0 & 0 & 0 & 0 & 0 & 0 & 1  \\
                                  -2 & 0&2&-1&4&1&-2&1\\
                                  0&-2&1&1&-1&5&-1&-1\\
                                  1&-1&-1&1&-2&2&2&-1\\
                                  1&0&-1&0&-2&0&1&1\end{pmatrix}.$$
The entropy of the shift is $2 m(t^2-4t+1)$, approximately $2.63392$.                             
\end{example}

\begin{example} The group of $k=5_2$ has presentation
$\< x, y \mid xyx\bar y \bar x y x =  y xyx\bar y \bar x y\>$ and Riley
polynomial $\Phi_{7,3}(w) = w^3+ w^2+ 2 w +1$.  The based $\g_\Phi$-twisted Alexander module $\A_{\g_\Phi}^0$ has a presentation of the form 
$$\<1\otimes v \otimes y \mid t^2 \otimes v P \otimes y = 
t \otimes v Q \otimes y + 1 \otimes v R \otimes y \>,$$
where $v$ ranges over a basis for $\Z^6$, and $P, Q$ and $R$ are the $6 \times 6$ integral matrices
$$P= XYX\bar Y+ Y X \bar Y \bar X Y X,\ Q= -I-Y X \bar Y \bar X,\ R= X + XYX\bar Y \bar X+Y X \bar Y.$$

The module can be expressed as
$$\cdots  \oplus_{\Z^{12}} \Z^{12} \oplus_{\Z^{12}} \Z^{12} \oplus_{\Z^{12}} \cdots,$$
where the amalgamating maps $g$ and $f$ are given by right multiplication by the matrices
$$ G = \begin{pmatrix} 0 & I\\ R& Q \end{pmatrix},\  F= \begin{pmatrix} I & 0 \\ 0 & P \end{pmatrix}.$$
 The images of $f$ and $g$ have index 25, and the total twisted Alexander polynomial is 
$$\D_{k, \g_\Phi} = {\rm det}(G-t F) = 25 t^6 - 104 t^5+ 219 t^4- 272 t^3+ 219 t^2 - 104 t + 25.$$
If replace $\Z$ by $\Z[1/5]$, then $f, g$ are invertible.  The dynamical system
$({\rm Col}_{\g_\Phi}^0,\s)$ is conjugate to an automorphism of the 12-dimensional solenoid $\widehat {\Z[1/5]^{12}}$ with entropy $m(\D_{k, \g_\Phi})$, approximately
$\log 25 + \log 1.82996 = 3.82317.$

 \end{example}

\begin{example} Riley concludes \cite{riley1} with ``an account of the representations of our favourite knot," a knot $k$ that ``has a confluence of noteworthy properties." Riley observed, for example, that $k$ is  the equatorial cross-section  of an unknotted 2-sphere in ${\mathbb R}^4$. His guess that $k$ is S-equivalent to the square knot was confirmed in \cite{hsi}. Hence abelian invariants (untwisted Alexander polynomials, homology of cyclic and branched covers) cannot distinguish $k$ from the square knot. 

A diagram of $k$ with Wirtinger generators for its group appears in Figure \ref{rileyknot}. Riley noted that, for any complex number $w$, the assignment
$$x_0\mapsto\begin{pmatrix} 1 & 1\\  0 & 1\end{pmatrix},\ x_1\mapsto \begin{pmatrix} 1 & 0 \\ -1 & 1\end{pmatrix}$$
$$x_2\mapsto\begin{pmatrix} 1-w & w^2\\  -1 & w+1\end{pmatrix},\ x_3\mapsto \begin{pmatrix} 1 & 0 \\ -w^2 & 1\end{pmatrix},$$
determines a nonabelian parabolic representation $\g$. 

\begin{figure}
\begin{center}
\includegraphics[height=2.5 in]{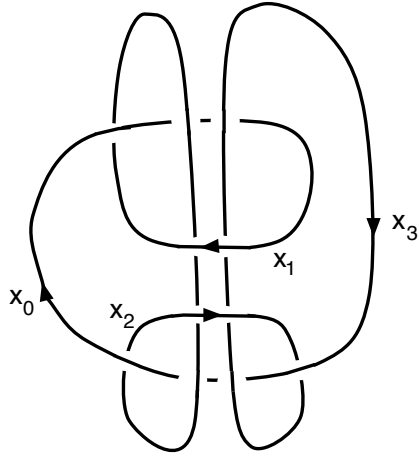}
\caption{Riley's knot}
\label{rileyknot}
\end{center}
\end{figure}

For any $w$, the determinant of the $3 \times 3$ Jacobian matrix
$$\biggl ( {  {\partial x_i} \over {\partial x_j} }\biggr)_{1 \le i, j \le 3}$$
with entries evaluated at $\g$ gives the same result, $(t-1)^6(t^2+1)$. If we let $w$ be an integer, then we obtain a (total) representation in ${\rm SL}_2 \Z$. The $\g$-twisted Alexander polynomial of $k$ is $\D_{k, \g}(t) = (t-1)^4 (t^2+1)$. We show that this polynomial does not arise from the square knot, for any nonabelian parabolic representation of its group.

Consider the square knot $k'$ in Figure \ref{squareknot}, with Wirtinger generators 
$x_0, x_1, x_2$ indicated. Its group has presentation 
$$\<x_0, x_1, x_2 \mid x_1 x_0 x_1 = x_0 x_1 x_0,\  x_2 x_0 x_2 = x_0 x_2 x_0 \>.$$ 
As shown in \cite{riley1}, any nonabelian parabolic representation is conjugate to one of the form:
$$x_0\mapsto\begin{pmatrix} 1 & 1\\  0 & 1\end{pmatrix},\ x_1\mapsto \begin{pmatrix} 1 & 0 \\ w & 1\end{pmatrix}, \ x_2\mapsto\begin{pmatrix} 1+u v & v^2\\  -u^2 & 1-u v\end{pmatrix}.$$
One easily verifies that $w=-1$ while $u=\pm 1$ or $u=0$. 
If $u=\pm 1$, then $v$ is arbitrary, and the twisted Alexander polynomial is 
$(t- 1)^2(t^2+1)^2$.  
If $u=0$,  then $v= 1$, and the polynomial is $(t^2-t+1)^2(t^2+1)$.
\end{example}
\begin{figure}
\begin{center}
\includegraphics[height=1.5 in]{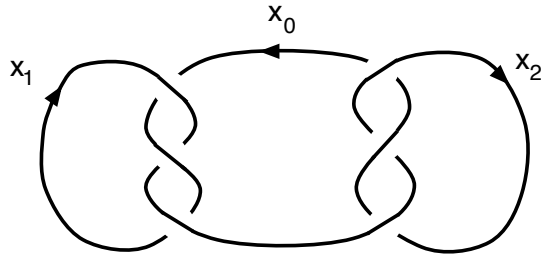}
\caption{Square knot}
\label{squareknot}
\end{center}
\end{figure}

\begin{example} \label{whiteheadex} Consider the Whitehead link $\ell$ in Figure \ref{whitehead} with Wirtinger generators indicated. Its group $\pi$ has presentation 
$\< x, y \mid  w y = y w\>,$
where $w= \bar x \bar y x y x \bar y \bar x.$ 

The associated Riley polynomial 
is $\Phi(w)= w^2+ 2 w +2$. Let $\g_\Phi: \pi \to
{\rm GL}_4 \Z$ denote the total representation. It is a straightforward
matter to check that the based $\g_\Phi$-twisted Alexander module is generated by the $1$-chains $1 \otimes e_i \otimes y\  (1\le i \le 4)$ with relators
\begin{equation*}\begin{split}
-1\otimes &v \bar X \bar Y \otimes y+
t_1 \otimes v \bar X \bar Y X \otimes y - 
t_1^2 \otimes v \bar X \bar Y XY X \bar Y \otimes y \\
&+ t_1 \otimes v \bar X \bar Y X Y X \bar Y \bar X \otimes y -
 t_1 t_2 \otimes v \otimes y + t_2 \otimes  v Y\bar X \bar Y \otimes y\\
 & - t_1 t_2 \otimes v Y \bar X \bar Y X \otimes y + t_1^2 t_2 \otimes v Y\bar X \bar Y X Y X \bar Y \otimes y.
\end{split} \end{equation*}

Proposition \ref{finite} below implies that $H_0(X', V)$ is finite. Hence Proposition \ref{alexcompute} enables us to  compute the $\g_\Phi$-twisted Alexander polynomial of $\ell$. It is given by 
the following array in which the number in position $i, j$ is the coefficient of $t_1^i t_2^j.$
$$\begin{array} {ccccc} 1 & -4 & 6 & -4 &1\\ 
                                      -4 & 12 & -16 & 12 & -4 \\
                                      6 & -16 & 28 & -16 & 6 \\
                                      -4 & 12 & -16 & 12 & -4\\
                                      1 & -4 & 6 & -4 & 1 \end{array}$$
                                      
By comparison, the untwisted Alexander polynomial is $(t_1 -1)(t_2-1)$, with coefficient array:
$$\begin{array} {cc} -1 & 1 \\ 1 & -1 \end{array}$$
\end{example}

\begin{figure}
\begin{center}
\includegraphics[height=2 in]{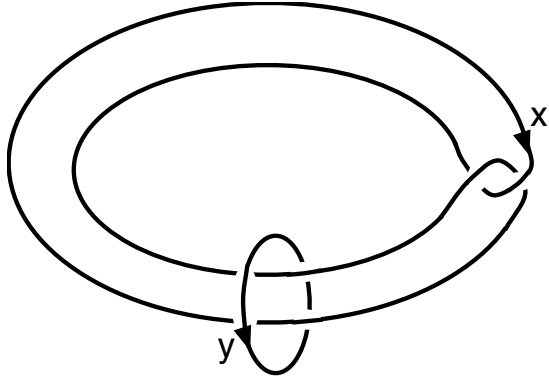}
\caption{Whitehead link}
\label{whitehead}
\end{center}
\end{figure}

While the untwisted Alexander polynomial is a product of cyclotomic polynomials, with trivial Mahler measure, the twisted polynomial is not.
The logarithmic Mahler measure of $\D_{\ell, \g_\Phi}$ is approximately 2.5. In particular, Theorem \ref{mainthm} ensures the growth of torsion of $\g_\Phi$-twisted homology in the cyclic covers of the link complement.

\section{Evaluating the twisted Alexander polynomial for 2-bridge knots and links}

The following establishes the nontriviality of twisted Alexander polynomials for any 2-bridge knot and nonabelian parabolic representation of that group. 

\begin{theorem}\label{eval} Let $k$ be a 2-bridge knot, and let $\g_\phi$ be a total representation associated to a nonabelian parabolic representation $\g: \pi \to {\rm SL}_2{\mathbb C}$. Then $|\D_{k, \g_\phi}(1)| = 2^{{\rm deg}\,\phi}$. \end{theorem}

\begin{cor} \label{nontrivial} Let $k$ be a $2$-bridge knot. For any nonabelian parabolic representation $\g$, the $\g$-twisted Alexander polynomial $\D_{k, \g}(t)$ is nontrivial. \end{cor}

\begin{proof} Assume after conjugation that $\g$ has the standard form given by (\ref {standard}) above. Let $\phi$ be the minimal polynomial of $w$, an irreducible factor of the Riley polynomial. 
Denote its roots by $w_1= w, w_2, \ldots, w_m$, where $m$ is the degree of $\phi$. The Galois group of $\Q[w]$ over $\Q$ is transitive on the roots. It follows that if $\D_{k, \g}(t)$ is trivial, then it is trivial 
whenever any $w_i, (i>1)$ is substituted for $w$ in (\ref{standard}). However, the total twisted polynomial $\D_{k, \g_\phi}(t)$ is the product of such polynomials. Hence $\D_{k, \g}(t)$ is nontrivial. 
\end{proof}

\begin{remark} Corollary \ref{nontrivial} implies that the Wada invariant of any 2-bridge knot corresponding to a nonabelian parabolic representation is nontrivial. 

Whether or not the Wada invariant of a knot corresponding to any nonabelian representation is nontrivial is unknown (cf. \cite{morifuji}, \cite{friedl}). One might suspect a negative answer in view of  M. Suzuki's result \cite{suzuki} that the Lawrence-Krammer representation of the braid group $B_4$,  a faithful nonabelian representation, yields a trivial Wada invariant.

\end{remark} 

\begin{conj} \label{Conj} For any $2$-bridge knot $k$ and total representation $\g_\phi$ as in Theorem \ref{eval}, $|\D_{k, \g_\phi}(-1)| = 2^{{\rm deg}\,\phi} \mu^2,$ for some integer $\mu$. \end{conj}

\begin{remark} Theorem \ref{eval} and Conjecture \ref{Conj} have been shown by M. Hirasawa and K. Murasugi \cite{hm} in the case that the group of $k$ maps onto a torus knot group, preserving meridian generators, and $\g$ is the pullback of a nonabelian parabolic representation. \end{remark}

In order to prove Theorem \ref{eval}, we consider an arbitrary $2$-bridge knot $k$ and total representation $\g_\phi$ associated to 
a nonabelian parabolic representation $\g: \pi \to {\rm SL}_2{\mathbb C}$. Denote the degree of $\phi$ by $m$. The following lemma follows from the proof of Corollary \ref{mainthmknots}. However, the direct argument below yields extra, useful information. 

\begin{lemma} \label{square} The $\Z[t^{\pm 1}]$-module $H_1(X'; V)$ has a square matrix presentation. \end{lemma}

\begin{proof} Consider the chain complex (\ref{chaincplx}):
$$0 \to C_2(X'; V)\ {\buildrel \partial_2 \over \longrightarrow}\   C_1(X'; V)\  {\buildrel \partial_1 \over \longrightarrow}\   C_0(X'; V)\ \to 0.$$
It can be written as 
$$0 \to \Z[t^{\pm 1}]^{2m} \ {\buildrel \partial_2 \over \longrightarrow}\    \Z[t^{\pm 1}]^{2m} \oplus  \Z[t^{\pm 1}]^{2m}\  {\buildrel \partial_1 \over \longrightarrow}\    \Z[t^{\pm 1}]^{2m}\ \to 0.$$
The module of $1$-chains is freely generated by 
$1 \otimes e_i \otimes x$ and $1 \otimes e_i \otimes y$, where $e_i$ ranges over a basis for $V = \Z^{2m}$. 

The $\partial_1$-images of the chains $a_i= 1 \otimes e_i \otimes (x - y)$ generate $C_0(X'; V)$ (see the discussion preceding Example \ref{trefoil}). For each $i = 1, \ldots, 2m$, choose a $\Z[t^{\pm 1}]$-linear combination $u_i$ of the chains $a_i$ such that $b_i = 1 \otimes e_i \otimes y - u_i$ is a cycle. Then $a_1, \ldots, a_{2m}, b_1, \ldots, b_{2m}$ is a second basis of $C_1(X'; V)$ such that $b_1, \ldots, b_{2m}$ freely generate the submodule of 1-cycles.
Hence $H_1(X'; V)$ is isomorphic to a quotient
$\Z[t^{\pm 1}]^{2m}/  \Z[t^{\pm 1}]^{2m} A$, for some square matrix $A$ of size $2m$. \end{proof} 

\begin{cor} \label{D(1)} $|\D_{k, \g_\phi}(1)|$ is equal to the order 
of the twisted homology group $H_1(X_1; V)= H_1(X; V)$. \end{cor}

\begin{proof} The argument above shows that $\partial_1$ is surjective. Hence $H_0(X'; V) =0$. A Wang sequence shows that 
$$H_1(X; V) \cong H_1(X'; V)/(t-1)H_1(X'; V).$$
From Lemma \ref{square}, the matrix $A(1)$ obtained from $A$ by setting $t=1$ presents $H_1(X; V)$. The order of $H_1(X; V)$ is $|{\rm det}\ A(1)|= |\D_{k, \g_\phi}(1)|.$
\end{proof} 

In view of the corollary, we focus our attention on $H_1(X; V)$. This finitely generated abelian group is determined by the 
$\g_\phi$-twisted chain complex
\begin{equation}\label{cplx} C_*(X; V) =V\otimes_{\Z[\pi]}C_*(\tilde X), \end{equation}
just as in Section \ref{twistedhomology} but with trivial $t$-action. 

The boundary of the knot exterior $X$ is an incompressible torus $T$, and the restriction of $\g_\phi$ to $\pi_1T$ determines $\g_\phi$-twisted absolute and relative homology groups $H_*(T; V_{\g_\phi})$ and $H_*(X, T; V)$. The latter groups are defined  
by the  quotient chain complex $C_*(X, T; V)=C_*(X; V)/C_*(T; V)$ (see, for example, \cite{kl}). A long exact sequence relating the absolute and relative homology groups is found from the natural short exact sequence of chain complexes
$$0 \to C_*(T; V) \to C_*(X; V) \to C_*(X, T; V) \to 0.$$

Lefschetz duality for twisted homology is presented in 
\cite{kl} (see also \cite{kv}). We apply a version in our situation.
Tensoring the chain complex \ref{cplx} with the field $\Z/(2)$, we obtain twisted homology groups that we denote by $H_*(X; V^{(2)})$ and $H_*(X, T; V^{(2)})$.  Let $\bar V$ be the 
module $V$ with right $\pi$-action given
$$v \cdot g = v (\g_\phi(g)^{-1})^\tau,$$
where $\tau$ denotes transpose. Substituting $\bar V$ for $V$ in the chain complex \ref{cplx}, and tensoring with $\Z/(2)$, yields a new complex with homology groups that we denote by 
$H_*(X; \bar V^{(2)})$ and $H_*(X, T; \bar V^{(2)})$. 

The inner product $\{\, , \}: V \times  \bar V \to \Z/(2)$  given by usual dot product is non-degenerate, and satisfies
$$\{r v, w\} = r \{v, w\} = \{v, rw\}$$
$$\{v \cdot g, w\} = \{ v, w g^{-1}\},$$
for all $r \in \Z/(2)$ and $g \in \pi$. (Readers of \cite{kl} should note the missing inverse sign in Equation (2.5).) As in Section 5.1 of \cite{kl}, we have $H_q(X; V^{(2)})\cong H_{3-q}(X, T; \bar V^{(2)})$ for all $q$. 

Taking inverse transpose of matrices $X, Y$ modulo 2 merely shifts the off-diagonal blocks to the opposite side. The ranks of the resulting homology groups are unaffected. We have

\begin{prop}\label{ranks} ${\rm rk}\,\, H_i(X; V^{(2)}) = 
{\rm rk}\,\,H_{3-i}(X, T; V^{(2)})$, for each $i$. \end{prop}

The following lemma will enable us to determine several boundary maps in appropriate chain complexes. 
Recall that the group $\pi$ of $k$ has a presentation of the form
$\<x, y \mid r\>$, where $r$ has the form $Wx =yW$. We denote the class of an oriented longitude by $l$. 

\begin{lemma} \label{partials} 
\begin{enumerate}
\item[\rm(1)] $\g_\phi(l)$ is a block matrix of $m \times m$ integer matrices  
$$L=\begin{pmatrix} - I & E \\  0 & - I\end{pmatrix},$$
where $I, 0$ are the identity and zero matrices, respectively, and $E$ has only even entries. 
\item[\rm(2)] The matrix of Fox partial derivative matrix $(\partial r/\partial x)^{\g_\phi}$ has the form 
$$\begin{pmatrix} 0 & *\\  0 & *\end{pmatrix},$$
where all blocks are $m \times m$. Similarly, 
$(\partial r/\partial y)^{\g_\phi}$ has (up to sign) the form 
$$\begin{pmatrix} I & 0\\  * & 0\end{pmatrix}.$$
\end{enumerate}
\end{lemma}

\begin{proof} Statement (1) follows immediately from Theorem 2 of \cite{riley1}. 

We prove statement (2). Fox's fundamental formula (2.3) of \cite{fox} implies that 
$$(r-1)^{\g_\phi} = (\partial r/ \partial x)^{\g_\phi}(X-I) +  (\partial r/ \partial y)^{\g_\phi}(Y-I).$$
The left hand side is trivial, since $r =1$ in $\pi$. Using the forms for $X,Y$, we have
$$\begin{pmatrix} 0 & 0\\  0 & 0\end{pmatrix} = (\partial r/ \partial x)^{\g_\phi}\begin{pmatrix} 0 & 1\\  0 & 0\end{pmatrix} +  (\partial r/ \partial y)^{\g_\phi}\begin{pmatrix} 0 & 0\\  C & 0\end{pmatrix}.$$
It follows immediately that $(\partial r/\partial x)^{\g_\phi}$ and $(\partial r/\partial y)^{\g_\phi}$ have the forms $\begin{pmatrix} 0 & *\\  0 & *\end{pmatrix}$ and $\begin{pmatrix} * & 0\\  * & 0\end{pmatrix},$
respectively. 

In order to see that the upper left-hand block of $(\partial r/\partial y)^{\g_\phi}$ is the identity matrix, we use form $Wx= yW$ of the relation $r$. Fox calculus yields
$$(\partial r/\partial y)^{\g_\phi} = (\partial W/\partial y)^{\g_\phi}-I - Y(\partial W/\partial y)^{\g_\phi}.$$
The right-hand side can be written as 
$$\begin{pmatrix} 0 & 0\\  -C & 0\end{pmatrix}(\partial W/\partial y)^{\g_\phi}-I,$$
and from this the desired form follows immediately.
\end{proof}

\begin{remark} Similar arguments show that, in fact, the upper right-hand block of $(\partial r/\partial x)^{\g_\phi}$ is unimodular. We will not require this fact here. \end{remark} 

We compute $H_i(T; V)$, for $i=0,1$. Beginning with the presentation $\pi_1T = \<x, l \mid xl=lx\>$, we construct a canonical $2$-complex with a single 0-cell *, $1$-cells $x$, $l$ corresponding to generators and a single $2$-cell $s$ corresponding to the relation. Consider the associated $\g_\phi$-twisted chain complex
$$0 \to C_2(T; V)\cong V\ {\buildrel \partial_2 \over \longrightarrow}\
C_1(T; V)\cong V \oplus V \ {\buildrel \partial_1 \over \longrightarrow}\
C_0(T; V)\cong V\to 0,$$
with 2-chain generators $v \otimes s$; 
$1$-chain generators $v\otimes x$ and $v \otimes l$; 
and $0$-chain generators $v\otimes *$. Here $v$ ranges freely over a basis $e_1, \ldots, e_{2m}$ for $V=\Z^{2m}$. 

We have $\partial_1(v \otimes x) = v\otimes \partial x = v\otimes (x * -*) = v(X-I)\otimes *$. Writing $v = (v_1, v_2) \in \Z^m \oplus Z^m$, we can express the $0$-chain boundaries obtained this way as $(0,v_1)\otimes *$, where $v_1$ is arbitrary. Similarly, $\partial_1(w\otimes l) = v(L-I)\otimes * =
(-2 w_1, w_1 E-2 w_2)\otimes *$, where $w=(w_1, w_2)\in V\oplus V$. 
It follows that $H_0(T; V) \cong (\Z/2)^m$, with generators
represented by the $0$-chains $e_i\otimes *$, where $1 \le i \le m$.

The computation of $H_1(X; V)$ is similar. From the previous description of $\partial_1(v\otimes x)$ and $\partial_1(w \otimes l)$, 
it is easy to see that $1$-cycles have the form \begin{equation}\label{1}(2 w_2, v_2)\otimes x + (0, w_2)\otimes l. \end{equation} 
Applying Fox calculus to the relation $xl=lx$, we find that $\partial_2(v \otimes s) =
v \otimes \partial_2 s = v\otimes (1-l)x + v\otimes (x-1)l =
v(I-L)\otimes x + v(X-I)\otimes l.$
By Lemma \ref{partials}, the latter expression reduces to 
\begin{equation}\label{2}(2v_1, -v_1 E+2 v_2)\otimes x + (0, v_1)\otimes l.  \end{equation} Comparing the expressions (\ref{1}) and (\ref{2}),
we see that $H_1(T; V) \cong (\Z/2)^m$, with 
generators represented by $1$-chains $e_i\otimes x$, where $m< i \le 2m$.

Consider part of the long exact sequence of $\g_\phi$-twisted homology groups associated to the pair $(X, T)$:
$$H_1(T; V) \to H_1(X; V) \to H_1(X, T; V) \to H_0(T; V) \to H_0(X; V)$$
Recall that $H_0(X; V)$ vanishes. 

We compute $H_1(X, T; V)$. Consider the following presentation of $\pi_1 X$:
$$\<x, y, l \mid q, r, s\>,$$
where $r$ and $s$ are given above, and $q$ is the relation $l=u$ expressing the longitude $l$ as  a word $u$ in  $x^{\pm 1}$ and $y^{\pm 1}$. An explicit word $u$ is not difficult to write (see \cite{riley1}), but we will not require this. The $\g_\phi$-twisted chain complex associated to this presentation contains the $\g_\phi$-twisted chain complex for $T$ as a subcomplex. For the purpose of computing homology below dimension 2, we may substitute this complex for $C_*(X; V)$.  

The relative complex $C_*(X, T; V)$ has the form
$$0 \to C_2(X, T; V)\cong V \oplus V\ {\buildrel \partial_2 \over \longrightarrow}\
C_1(X, T; V)\cong V \ \to\ 
0,$$
with relative 2-chain generators represented by $v \otimes r$ and $v \otimes s$; relative $1$-chain generators represented by $v \otimes y$. As before $v$ ranges freely over a basis for $\Z^{2m}$. 

We have 
$\partial_2(v \otimes r) = v \otimes (\partial r/\partial y) y = v(\partial r/\partial y) ^{\g_\phi} \otimes y$. By Lemma \ref{partials}, $\partial_2(v \otimes r)$ has the form $(v_1, 0) \otimes y$, where $v_1 \in \Z^m$ is arbitrary. Hence the rank of $H_1(X, T; V_{\g_\phi})$  cannot exceed $m$.
 
Similarly, $\partial_2(v \otimes s) = v \otimes (\partial s/\partial y) y = v(\partial s/\partial y) ^{\g_\phi} \otimes y$. Recall that the relation $s$ has the form 
$l = u$. $\partial_2(v \otimes s)= v(\partial u/\partial y) ^{\g_\phi} \otimes y$.  Hence by Fox's fundamental formula, 
$$(u-1)^{\g_\phi} = (\partial u/ \partial x)^{\g_\phi}(X-I) +  (\partial u/ \partial y)^{\g_\phi}(Y-I).$$
The left-hand side is equal to $(l-1)^{\g_\phi} = L-I$, since $u=l$ in $\pi$; by Lemma \ref{partials}, this matrix has the form
$$ \begin{pmatrix} -2 I & E\\  0 & -2 I\end{pmatrix}.$$
Writing 
$$(\partial u/\partial x)^{\g_\phi} = \begin{pmatrix} U_x^{11} & U_x^{12} \\ U_x^{21}  & U_x^{22} \end{pmatrix},\quad (\partial u/\partial y)^{\g_\phi} = \begin{pmatrix} U_y^{11} & U_y^{12} \\ U_y^{21}  & U_y^{22} \end{pmatrix}, $$
and using the forms for $X,Y$, the right-hand side has the form
$$\begin{pmatrix} 0 & U_x^{11}\\  0 & U_x^{21}\end{pmatrix}  +  \begin{pmatrix} U_y^{12}C & 0\\  U_y^{22}C & 0\end{pmatrix}.$$
Since $C$ is unimodular (see Remark \ref{companion}), we see that 
that $(\partial u/\partial y)^{\g_\phi}$ has the form
$$\begin{pmatrix} * &-2 C^{-1}\\  * & 0\end{pmatrix}.$$
Hence
$\partial_2(v \otimes s)$ has the form $(*, w_2)\otimes y$ , where $w_2$ ranges over $(2\Z)^m$ as $v$ varies over $\Z^{2m}$. It follows that $H_1(X, T; V)\cong (\Z/2)^m$, with generators represented by $e_i \otimes y,$ where $m< i\le 2m.$ 

We see also that the map $H_1(X, T; V) \to H_0(T; V)$ in the long exact sequence of the pair $(X, T)$ is an isomorphism. From this it follows that the map $H_1(T; V) \to H_1(X; V)$ is surjective. Since $H_1(T; V)$ is 2-torsion, so is $H_1(X; V)$. 

We can now prove Theorem \ref{eval}. 

\begin{proof} 
By Lefschetz duality, 
$${\rm rk}\, H_2(X; V^{(2)})={\rm rk}\, H_1(X, T; V^{(2)})= m.$$
The universal coefficient theorem for homology implies that $H_0(X; V^{(2)})\cong H_0(X; V)\otimes \Z/(2)$. Since $H_0(X; V)$ is  trivial,  ${\rm rk}\,H_0(X; V^{(2)}) =0$. Also, ${\rm rk}\, H_3(X; V^{(2)})= {\rm rk}\, H_0(X, T; V^{(2)}) =0$. 

The alternating sum 
${\rm rk}\, H_0(X; V^{(2)})- {\rm rk}\, H_1(X; V^{(2)})+ {\rm rk}\, H_2(X; V^{(2)})- {\rm rk}\, H_3(X; V^{(2)})$ coincides with the Euler characteristic of $X$, which is 0. Hence ${\rm rk}\, H_1(X; V^{(2)})=m$.   By the universal coefficient theorem for homology, $H_1(X; V^{(2})  \cong H_1(X; V)\otimes \Z/2$. Since $H_1(X; V)$ is $2$-torsion, it is isomorphic to $(\Z/(2))^m$. Therefore
$|\D_{k, \g_\phi}(1)|= |H_1(X; V)|= 2^m$. \end{proof} 

In the beginning of Section \ref{examples} we saw that for a $2$-bridge knot $k$ with exterior $X$, the group $H_0(X'; V)$ vanishes for every 
total representation $\g_\phi$ associated to a nonabelian parabolic representation of the group of $k$. The result depends on the fact that the roots $w$ of any Riley polynomial are algebraic units. 

For $2$-bridge links, the algebraic integer $w$ need not be a unit (see Example \ref{whiteheadex}). Nevertheless, we have the following, which is useful for computation in view of Proposition \ref{alexcompute}. 

\begin{prop}\label {finite} Assume that $\ell$ is  a $2$-bridge link with exterior $X$. Let $\g_\phi$ be a total representation associated to a nonabelian parabolic ${\rm SL}_2 \C$ representation of the group of $\ell$. Then 
$H_0(X'; V)\cong \Z/(\phi(0)) \oplus \Z/(\phi(0))$. In particular, the group is finite. \end{prop}

\begin{proof} We use notation similar to that above. Beginning with the presentation $ \<x, y \mid Wx=xW\>$ for the link group, we construct a canonical $2$-complex with a single 0-cell *, $1$-cells $x, y$ corresponding to generators, and a single $2$-cell corresponding to the relation. We consider the associated $\g_\phi$-twisted chain complex
$$C_*(X'; V) = (\Z[\Z^2]\otimes V)\otimes_{\Z[\pi]}C_*(\tilde X),$$
with $1$-chain generators $1 \otimes v\otimes x$ and $1 \otimes v \otimes y$; 
and $0$-chain generators $1 \otimes v\otimes *$. As before, $v$ ranges freely over a basis $e_1, \ldots, e_{2m}$ for $V=\Z^{2m}$, where $m$ is the degree of $\phi$. We will not need to consider $2$-chains.

The $\g_\phi$-twisted homology group $H_0(X'; V)$ is the quotient of $\Z[\Z^2]\otimes V$,  the $0$-chains, modulo the image of the boundary operator $\partial_1$. Since $\partial_1(1\otimes v\otimes x) = 
t_1\otimes vX \otimes * - 1\otimes v \otimes *$ and 
$\partial_1(1\otimes v\otimes y) = 
t_2\otimes vY \otimes * - 1\otimes v \otimes *$, it follows immediately that $t_1^n \otimes v \otimes * = 1 \otimes v X^{-n}\otimes *$ and 
$t_2^n \otimes v \otimes * = 1 \otimes v Y^{-n} \otimes *$ for all $n \in \Z$. We can define an abelian group homomorphism 
$$f: \Z[\Z^2]\otimes V/(\im \partial_1) \to V/V(XY-YX)$$
by sending $p(t_1, t_2) \otimes v \otimes * \mapsto v p(\bar X, \bar Y)$, for any polynomial $p \in \Z[\Z^2]$, and extending linearly. The assignment $v \mapsto 1 \otimes v \otimes *$, for $v \in V$, 
induces an inverse homomorphism that is well defined, since
$$1 \otimes vXY \otimes * = t_2^{-1}t_1^{-1}\otimes v \otimes *=t_1^{-1}t_2^{-1}\otimes v \otimes *= 1 \otimes vYX\otimes * $$
for all $v\in V$. Hence $f$ is an isomorphism. 

The matrix
$XY-YX$ has the form
$$\begin{pmatrix} C& 0\\ 0 & -C\end{pmatrix},$$
where $C$ is the $m \times m$-companion matrix of $\phi$ and $0$
is the zero matrix of that size. The determinant of $XY-YX$ is 
$({\rm det}\ C)^2$, which is the square of the constant coefficient of $\phi$. It follows that $H_0(X'; V) \cong \Z/({\phi(0)}) \oplus \Z/({\phi(0)})$.

\end{proof} 

\section{Twisted Alexander polynomials of torus knots}

The group $\pi$ of any fibered knot has a presentation of the form
$$\<x,  a_1, \ldots, a_{2g} \mid x a_i x^{-1} = \mu (a_i)\ (1 \le i \le 2g)\>,$$ where $a_1, \ldots, a_{2g}$ are free generators of the fundamental group of the fiber $S$, and $\mu: \pi_1 S \to \pi_1 S$ is the automorphism induced by the monodromy. 

Let $\g: \pi \to {\rm GL}_N R$ be a representation, where $R$ is a Noetherian UFD. The twisted homology computations in Section \ref{twistedhomology} can be performed using the above presentation. We find that the based $\g$-twisted Alexander module $\A_\g^0$ is a free $R$-module of rank $2g N$, with generators $1 \otimes v \otimes a_i$, where $v$ ranges over a basis for $R^N$. Multiplication by $t$ induces an automorphism described by an $N\times N$-block matrix $M$ with $2g \times 2g$ blocks  
$$ X^{-1} \Biggl({ {\partial \mu(a_i)}\over{\partial a_j} }^\g\Biggr).$$
As before, $X$ is the image of the meridian $x$, and the notation ${}^\g$ indicates that the terms of the Fox partial derivatives are to be evaluated by $\g$.  The characteristic polynomial is the $\g$-twisted Alexander polynomial of $k$. 

In \cite{riley1}, Riley describes nonabelian parabolic representations of any torus knot $k$. The $(2n+1,2)$ torus knots are 2-bridge knots with $\a = 2n+1, \b=1$. Their Riley polynomials $\Phi_{2n+1, 1}(w)$ can be found recursively:
$\Phi_{1,1}(w)=1, \Phi_{3,1}(w)=w+1$, and 
$$ \Phi_{2n+5, 1}(w) = -\Phi_{2n+1, 1}(w) + (w+2) \Phi_{2n+3,1}(w), \quad n \ge 0.$$
For arbitrary torus knots, the situation is more complicated. Nevertheless, Riley shows that any parabolic representation for a torus knot group is conjugate to a representation with image in 
${\rm SL}_2 {\mathbb A}$, where $\mathbb A$ is the ring of algebraic integers. Consequently, any parabolic representation $\g$ determines a total representation $ \g_\phi$ and hence a total twisted Alexander polynomial $\Delta_{k, \g_\phi}(t)$, as above. 

\begin{theorem} \label{torusknot} For any parabolic representation $\g$ of the group of a torus knot $k$, the total twisted Alexander polynomial $\D_{k, \g_\phi}(t)$ is a product of cyclotomic polynomials. 

\end{theorem}

\begin{proof}  The monodromy $\mu$ of a $(p,q)$-torus knot is periodic of order $pq$. Moreover, $\mu^{pq}$ is inner automorphism by the longitude $l \in \pi'$. Hence, for any generator $a_i$ of the fundamental group of the fiber $S$, as above, we have  
$$\mu^{pqr}(a_i) = l^r \mu(a_i) l^{-r}.$$
Let $M$ be the matrix describing the action of $t$ on $\A_{\g_\phi}^0$, as above. A standard calculation using Fox calculus shows that the $ij$th block of
$$ \Biggl({ {\partial \mu^{pqr}(a_i)}\over{\partial a_j} }^{\g_\phi}\Biggr)$$
is equal to 
$$(I+L +\cdots + L^{r-1}){\partial l \over \partial a_j}^{\g_\phi} + L^r {\partial \mu(a_i)\over \partial a_j}^{\g_\phi}- L^r A_i (L^{-1} + \cdots + L^{-r}) {\partial l \over \partial a_j}^{\g_\phi},$$
where $A_i$ denotes $\mu(a_i)$ and $L=\g_\phi(l)$. The sum can be rewritten as 
$$(I - L^r A_i L^{-r})(I + L + \cdots +L^{r-1}){\partial l \over \partial a_j}^{\g_\phi} + L^r {\partial \mu(a_i) \over \partial a_j}^{\g_\phi}.$$
Recall from Lemma \ref{partials} that $L$ has the form
$$\begin{pmatrix} -I & E \\ 0 & -I \end{pmatrix}.$$ Hence the entries of $L^r$ have polynomial growth as $r$ goes to infinity. Consequently, the entries of $M^r$ have polynomial growth. It follows that the spectral radius of $M$ is $1$. Since the characteristic polynomial, which is
$\D_{k, \g_\phi}(t)$, has integer coefficients, it is a product of cyclotomic polynomials by Kronecker's theorem. 
\end{proof}

\begin{remark} Hirasawa and Murasugi  \cite{hm} have  shown that
the total twisted Alexander polynomial of a $(2q+1,2)$-torus knot $k$
has the form 
$$\D_{k, \g_\Phi}(t) = (t^2+1) (t^{4q+2}+1)^{q-1},$$
a product of cyclotomic polynomials, as predicted by Theorem \ref{torusknot}. \end{remark}

Since the $\g$-twisted Alexander polynomial $\D_{k, \g}(t)$, which is well defined up to a unit in $\C[t^{\pm 1}]$,  divides the total twisted Alexander polynomial $\D_{k, \g_\phi}(t)$, the following is immediate. 

\begin{cor} For any parabolic representation $\g$ of a torus knot $k$, the $\g$-twisted Alexander polynomial is a product of cyclotomic polynomials.  
\end{cor} 

With minor modifications, the above arguments apply for any  permutation representation $\g: \pi \to {\rm GL}_N{\mathbb Z}$.

\begin{cor} For any finite-image permutation representation $\g$ of a torus knot $k$, the $\g$-twisted Alexander polynomial is a product of cyclotomic polynomials.  
\end{cor}


 \bigskip

\noindent {\sl Address for both authors:} Department of Mathematics and  Statistics, ILB 325, University of South Alabama, Mobile AL  36688 USA \medskip

\noindent {\sl E-mail:} silver@jaguar1.usouthal.edu; swilliam@jaguar1.usouthal.edu

\end{document}